\documentclass{aic}

\PassOptionsToPackage{obeyspaces}{url}
\usepackage{bookmark}
\bookmarksetup{open,numbered,depth=5}
\usepackage{amsfonts}
\usepackage{amsmath}
\usepackage{amssymb, amsthm, enumitem}
\usepackage{tikz}
\usepackage{bm}
\usepackage{dsfont} 
\usetikzlibrary{calc}

\patchcmd{\thebibliography}{\leftmargin\labelwidth}{\leftmargin\labelwidth\addtolength\itemsep{-0.1\baselineskip}}{}{}

\usepackage{mathtools}
\usepackage{xstring} 

\usepackage[nameinlink]{cleveref}

\usepackage[symbol]{footmisc}

\newtheorem{theorem}{Theorem}
\newtheorem{lemma}[theorem]{Lemma}

\newtheorem{observation}[theorem]{Observation}
\newtheorem{proposition}[theorem]{Proposition}
\newtheorem*{claim}{Claim}
\newtheorem{question}{Question}

\newcommand*{\eqdef}{\stackrel{\mbox{\normalfont\tiny def}}{=}} 
\newcommand*{\veps}{\varepsilon}                                
\DeclarePairedDelimiter\abs{\lvert}{\rvert}                     

\newcommand*{\p}{\mathbb{P}}
\newcommand*{\E}{\mathbb{E}}                                    
\newcommand*{\Z}{\mathbb{Z}}                                    
\newcommand*{\R}{\mathbb{R}}                                    
\newcommand*{\cP}{\mathcal{P}}

\newcommand*{\cN}{\mathcal{N}}

\newcommand*{\bC}{\bm{C}}
\newcommand*{\bF}{\bm{F}}
\newcommand*{\cF}{\mathcal{F}}
\newcommand*{\cO}{\mathcal{O}}
\newcommand*{\cC}{\mathcal{C}}

\newcommand*{\wu}{\widetilde{u}}
\newcommand*{\wv}{\widetilde{v}}

\DeclareMathOperator{\conv}{conv}
\DeclareMathOperator{\diam}{diam}

\newcommand*{\ipi}{\pi^{-1}}
\DeclareMathOperator{\spa}{span}
\DeclareMathOperator{\str}{stret}                               
\DeclareMathOperator{\vol}{vol}
\DeclareMathOperator{\area}{area}
\DeclareMathOperator{\covol}{covol}                             
\DeclareMathOperator{\SO}{SO}

\crefname{enumi}{step}{steps}
\crefname{part}{part}{parts}

\allowdisplaybreaks[4]

\aicAUTHORdetails{%
  title = {Convex Polytopes in Restricted Point Sets in \texorpdfstring{$\R^d$}{R to the d}}, 
  author = {Boris Bukh and Zichao Dong},
  plaintextauthor = {Boris Bukh, Zichao Dong},
   keywords = {Erd\H{o}s--Szekeres problem, convex cups and caps, Horton sets.},
}   

\aicEDITORdetails{%
   year={2025},
   number={1},
   received={27 May 2022},   
   published={29 January 2025},  
   doi={10.19086/aic.2025.1},      
}   

\begin{document}

\begin{frontmatter}[classification=text]

\title{Convex Polytopes in Restricted Point Sets in $\R^d$} 

\author[bukh]{Boris Bukh\thanks{Supported in part by Simons Foundation Fellowship and U.S. taxpayers through NSF CAREER grant DMS-1555149.}}
\author[dong]{Zichao Dong\thanks{Supported by the Institute for Basic Science (IBS-R029-C4).}}

\begin{abstract}
    For any finite point set $P \subset \R^d$, denote by $\diam(P)$ the ratio of the largest to the smallest distances between pairs of points in $P$. Let $c_{d, \alpha}(n)$ be the largest integer $c$ such that any $n$-point set $P \subset \R^d$ in general position, satisfying $\diam(P) < \alpha\sqrt[d]{n}$, contains a $c$-point convex independent subset. We determine the asymptotics of $c_{d, \alpha}(n)$ as $n \to \infty$ by showing for $\alpha \ge 2$ the existence of positive constants $\beta = \beta(d, \alpha)$ and $\gamma = \gamma(d)$ such that 
    \[
    \beta n^{\frac{d-1}{d+1}} \le c_{d, \alpha}(n) \le \gamma n^{\frac{d-1}{d+1}}. 
    \]
\end{abstract}
\end{frontmatter}

\section{Introduction} \label{sec:intro}
	
	\paragraph{Background.} A point set $P \subset \R^d$ is \emph{in general position} if no $d+1$ points from $P$ lie on the same $(d-1)$-dimensional hyperplane. A point set $Q$ is called \emph{convex independent} if $Q$ is in general position and the points from $Q$ are the vertices of a convex polytope. 
	
	The following natural question was asked by Esther Klein (later Mrs.~Szekeres): 
	
	\begin{question} \label{ques:es}
		Given a positive integer $n$, let $N = N(n)$ be the smallest positive integer such that every $N$-point set $P$ in general position in the plane contains an $n$-point convex independent set. Does $N$ exist? If so, how big is $N(n)$? 
	\end{question}
	
	In \cite{erdos_szekeres}, Erd\H{o}s and Szekeres proved the existence of $N$ by showing $N(n) \le \binom{2n-4}{n-2}+1$ via what is now well-known as the \emph{cup-cap argument}. They also proved in the same paper that $N(n) \ge 2^{n-2}+1$ and conjectured that this is sharp. This conjecture remains unsolved, with the best known upper bound $N(n) \le 2^{n(1+o(1))}$ due to Suk in \cite{suk}. Various generalizations and variations on \Cref{ques:es} have been studied over the years. For an extensive survey on the Erd\H{o}s--Szekeres problem, see \cite{morris_soltan}. A recent notable result is due to Pohoata and Zakharov \cite{pohoata_zakharov}, showing a bound of the form $2^{o(n)}$ for the analogous problem in $3$ dimensions.
	
	Define the \emph{normalized diameter} of a finite set $P\subset \R^d$ as
	\[
	\diam(P) \eqdef \frac{\max\{|a-b|: a, b \in P, \, a \neq b\}}{\min\{|a-b|: a, b \in P, \, a \neq b\}}, 
	\]
	where $|a-b|$ denotes the Euclidean distance between points $a$ and $b$ in $\R^d$. In this paper we study the following variant of \Cref{ques:es}: 
	
	\begin{question} \label{ques:res_es}
		Given a positive integer $n$, let $c = c_{d, \alpha}(n)$ be the biggest positive integer such that every $n$-point set $P$ in general position in $\R^d$ with $\diam(P) \le \alpha\sqrt[d]{n}$ contains a $c$-point convex independent set. What is the order of $c_{d, \alpha}(n)$ as $n$ goes to infinity? 
	\end{question}
	
	The results of Erd\H{o}s, Szekeres and Suk mentioned above imply that
	\[
	\lim_{\alpha \to \infty} c_{2,\alpha}(n) = \bigl(1 + o(1)\bigr) \log_2 n. 
	\]
	
	The $\alpha\sqrt[d]{n}$ upper bound on $\diam(P)$ is natural. To see this, we assume without loss of generality that the minimum distance between a pair of distinct points in $P$ is $1$. Place an open ball $U_i$ of radius $\frac{1}{2}$ centered at $p_i$ for each $p_i \in P$. Then the assumption on the minimum pairwise distance implies that $U_1, \dots, U_n$ are pairwise disjoint, and are contained in the big open ball $U$ of radius $\diam(P)+\frac{1}{2}$ centered at an arbitrary fixed point $p \in P$. Thus, a comparison of volumes shows that
	\[
	n \cdot \Omega \biggl( \Bigl(\frac{1}{2}\Bigr)^d \biggr) = \sum_{i=1}^n \vol(U_i) \le \vol(U) = O \biggl( \Bigl( \diam(P) + \frac{1}{2} \Bigr)^d \biggr), 
	\]
	and hence $\diam(P) = \Omega\bigl(\sqrt[d]{n}\bigr)$. 
	
	\Cref{ques:res_es} was first studied in the plane by Alon, Katchalski, and Pulleyblank in \cite{alon_katchalski_pulleyblank}. They proved that there exists an absolute constant $\beta = \beta(\alpha) > 0$ such that $\beta n^{\frac{1}{4}} \le c_{2, \alpha}(n) \le n^{\frac{1}{3}+o(1)}$ when $\alpha \ge 4$. Later on, both bounds were improved by Valtr in \cite{valtr92}. He showed that there exist absolute constants $\beta = \beta(\alpha) > 0$ and $\gamma > 0$ such that $\beta n^{\frac{1}{3}} \le c_{2, \alpha}(n) \le \gamma n^{\frac{1}{3}}$ when $\alpha \ge 1.06$. 
	
	The algorithmic version of the problem in $\R^3$ was studied contemporaneously with the present paper by Dumitrescu and T\'oth \cite{dumitrescu_toth}.  They also independently showed that
	$c_{3,\alpha}=\Theta_{\alpha}(n^{1/2})$.
	
	\paragraph{Results and paper organization.} Our main result is that $\beta n^{\frac{d-1}{d+1}} \le c_{d, \alpha}(n) \le \gamma n^{\frac{d-1}{d+1}}$ holds for each $\alpha \ge 2$, where $\beta = \beta(d,\alpha) > 0$ and $\gamma = \gamma(d) > 0$ are absolute constants. This estimate determines the exact order of $c_{d, \alpha}(n)$, and hence generalizes the results in \cite{valtr92}. To be specific, we shall prove the following pair of theorems: 
	
	\begin{theorem} \label{thm:lowerbound}
		For any dimension $d \ge 2$ and any $\alpha > 0$, there exists a constant $\beta = \beta(d,\alpha) > 0$ such that any $n$-point set $P \subset \R^d$ in general position with $\diam(P) < \alpha \sqrt[d]{n}$ contains a convex independent subset $Q$ satisfying~$|Q| \ge \beta n^{\frac{d-1}{d+1}}$. 
	\end{theorem}
	
	\begin{theorem} \label{thm:upperbound}
		For any dimension $d \ge 2$, there exists a constant $\gamma = \gamma(d) > 0$ such that for any positive integer $n$ there exists an $n$-point set $P \subset \R^d$ in general position with $\diam(P) \le 2 \sqrt[d]{n}$ containing no convex independent subset of size~$\gamma n^{\frac{d-1}{d+1}}$. 
	\end{theorem}
	
	We shall prove \Cref{thm:lowerbound} in \Cref{sec:prf_lower}. The proof is an easy generalization of Valtr's proof \cite{valtr92} of the lower bound on~$c_{2,\alpha}(n)$. 
	
	The proof of the upper bound (\Cref{thm:upperbound}) is more interesting. Horton \cite{horton} constructed arbitrarily large sets in $\R^2$ that contain no $7$-holes, i.e., vertices of a convex $7$-gon containing no other points from the point set. Valtr's construction behind the upper bound on $c_{2,\alpha}(n)$, a perturbed $\sqrt{n} \times \sqrt{n}$ grid, was inspired by that of Horton. Valtr's sets also contain no $7$-holes, while they further satisfy $\diam(P) \le \alpha\sqrt{|P|}$ (whereas the normalized diameter of Horton's sets is superpolynomial \cite{barba}). 
	
	We shall generalize Valtr's construction to higher dimensions. To achieve this, we shall generalize planar Horton sets to higher dimensional \emph{oscillators} in \Cref{sec:osci}. To analyze the behavior of oscillators, we shall generalize the \emph{convex cups} and \emph{convex caps} defined in \cite{erdos_szekeres} to higher dimensions in \Cref{sec:cupcap,sec:osci}. Finally, we shall prove \Cref{thm:upperbound} in \Cref{sec:prf_upper}. 
	
	\smallskip
	
	Another work related to \Cref{ques:res_es} is that of Conlon and Lim \cite{conlon_lim}. They proved that there exist arbitrarily small perturbations of the grid $[n]^d$ that are $d^{O(d^3)}$-hole-free, where $[n] \eqdef \{1, 2,\dotsc, n\}$. They too perturb the grid by Horton-like sets, and much of the work lies in generalizing Horton sets to higher dimensions. At the end of \cite{conlon_lim}, Conlon and Lim proposed the question of whether their constructions can also be used to generalize Valtr's upper bound on $c_{2,\alpha}(n)$ derived in \cite{valtr92} to higher dimensions. 
	Though we establish the upper bound on $c_{d, \alpha}(n)$ via Horton-like perturbations of the grid, our perturbations are different
	from those in \cite{conlon_lim}.

	\section{Proof of \texorpdfstring{\Cref{thm:lowerbound}}{Theorem 1}} \label{sec:prf_lower}
	
	Recall that $P \subset \R^d$ is in general position satisfying $|P| = n$ and $\diam(P) < \alpha\sqrt[d]{n}$.  We begin the proof of \Cref{thm:lowerbound} by introducing
	several objects related to~$P$.

    \begin{center}
		\begin{tikzpicture}[scale = 0.6]
			\node at (0, -1.8) {\textbf{\hypertarget{figone}{Figure 1}:} The ball slice $S_v$ (in gray) in $\R^2$.};
			\clip (-8, -1.1) rectangle (8, 6.6);
			\draw [shift={(0.,0.)},line width=1pt,fill=gray]  plot[domain=0.8047677837348086:2.3368248698549845,variable=\t]({1.*5.*cos(\t r)+0.*5.*sin(\t r)},{0.*5.*cos(\t r)+1.*5.*sin(\t r)}) -- cycle;
			
			\draw [thin] (-3.5,3.6) -- (-5.3,3.6);
			\draw [thin, <->] (-5.1,3.6) -- (-5.1,5);
			\draw [thin] (-5.3,5) -- (0,5);
			\node at (-5.3, 4.35) {\scriptsize $h$};
			\draw [line width=0.8pt, dash pattern=on 4pt off 4pt] (0, 0) circle (5);
			\draw [fill=black] (0, 0) circle (2pt);
			\draw [fill=black] (0, 5) circle (2pt);
			\draw [thin,->] (0, 0) -- (4, 3);
			\node at (0, -0.4) {\scriptsize origin};
			\node at (0, 5.25) {\scriptsize $v$};
			\node at (1.9, 1.65) {\scriptsize $r$};
			\draw[thin] (3.4663930246614556, 3.603348359315042) -- (3.4663930246614556, 6.4);
			\draw[thin] (-3.4663930246614556, 3.603348359315042) -- (-3.4663930246614556, 6.4);
			\draw[thin,<->] (3.4663930246614556, 6.2) -- (-3.4663930246614556, 6.2);
			\node at (0, 6.42) {\scriptsize $2s$};
		\end{tikzpicture}
	\end{center}
    
	Fix parameters $r \eqdef \alpha \cdot n^{\frac{1}{d}}$ and $h \eqdef \alpha \cdot n^{\frac{1-d}{d(d+1)}}$. For $R > 0$, denote by $B_R$ the open ball centered at the origin of radius $R$, and by $\partial B_R$ the boundary of $B_R$.
	The \emph{ball slice} in $B_r$ with apex $v \in \partial B_r$ is $S_v \eqdef \bigl\{x \in B_r : \langle x, v \rangle > r(r-h)\bigr\}$, where $\langle \bullet, \bullet \rangle$ denotes the standard Euclidean inner product on $\R^d$. The boundary of $S_v$ is the union of a spherical cap centered at $v$ and a $(d-1)$\nobreakdash-dimensional round disk of radius $s \eqdef \sqrt{r^2 - (r-h)^2} = \sqrt{2rh} \cdot \bigl(1-o(1)\bigr)$. The distance between $v$ and any point on the boundary of the aforementioned round disk is $\sqrt{s^2 + h^2} = \sqrt{2rh}$. See \hyperlink{figone}{Figure 1} for an illustration. 
	
	We record some geometric relations involving ball slices. They are clear from \hyperlink{figtwo_dup}{Figure 2}.
	
	\begin{observation} \label{obs:slice_sandwich}
		For any point $v \in \partial B_r$, we have the following containments involving $S_v$: 
		\begin{enumerate}[label=(\roman*), ref=(\roman*)]
			\item \label{subobs:ball}  The ball slice $S_v$ is contained in the closed ball of radius $\sqrt{2rh}$ centered at $v$.
			\item \label{subobs:cylinder} The ball slice $S_v$ is contained in the cylinder sharing the same base disk with $S_v$ of height $h$.
			\item \label{subobs:cone} The ball slice $S_v$ contains the cone with the same base disk and of the same height $h$.
		\end{enumerate}
	\end{observation}
	
	\begin{center}
		\def\slicecommonpart{
			\clip (-5.5, -1.3) rectangle (5.5, 9);
			\draw [shift={(0.,0.)},line width=0.3pt,fill=gray]  plot[domain=0.8047677837348086:2.3368248698549845,variable=\t]({1.*5.*cos(\t r)+0.*5.*sin(\t r)},{0.*5.*cos(\t r)+1.*5.*sin(\t r)}) -- cycle;
			\draw [fill=black] (0, 5) circle (2pt);
			\node at (0, 5.25) {\scriptsize $v$};
			\draw[thin] (3.4663930246614556, 3.603348359315042) -- (3.4663930246614556, 0);
			\draw[thin] (-3.4663930246614556, 3.603348359315042) -- (-3.4663930246614556, 0);
			\draw[thin,<->] (3.4663930246614556, 0.2) -- (-3.4663930246614556, 0.2);
			\node at (0, -0.025) {\scriptsize $2s$};
		}
		\begin{tikzpicture}[scale = 0.55]
			\begin{scope}
				\slicecommonpart
				\draw [thin] (-3.5,3.6) -- (-5.3,3.6);
				\draw [thin, <->] (-5.1,3.6) -- (-5.1,5);
				\draw [thin] (-5.3,5) -- (0,5);
				\draw[thin, ->] (0, 5) -- (2.99, 7.24);
				\node at (1.1, 6.5) {\scriptsize $\sqrt{2rh}$};
				\node at (-5.3, 4.35) {\scriptsize $h$};
				\draw [line width=1pt] (0, 5) circle (3.73718);
				\node at (0,-1) {a) Ball};
			\end{scope}
			\begin{scope}[xshift=8.5cm]
				\slicecommonpart
				\draw [line width=1pt] (-3.4663930246614556, 3.603348359315042) -- (-3.4663930246614556, 5) -- (3.4663930246614556, 5) -- (3.4663930246614556, 3.603348359315042) -- cycle;
				\node at (0,-1) {b) Cylinder};
			\end{scope}
			\begin{scope}[xshift=17cm]
				\slicecommonpart
				\draw [line width=1pt,line join=bevel] (-3.4663930246614556, 3.603348359315042) -- (0, 5) -- (3.4663930246614556, 3.603348359315042) -- cycle;
				\node at (0,-1) {c) Cone};
			\end{scope}
			\node at (8.5, -2.1) {\textbf{\hypertarget{figtwo_dup}{Figure 2}:} Two supersets and one subset of $S_v$.};
		\end{tikzpicture}
	\end{center}
	
	\vspace{-1em}
	Set $\hat{S} \eqdef S_{re_1} - re_1$, the translate of the ball slice $S_{re_1}$ by vector $-re_1$, where $e_1 \eqdef (1, 0, \dotsc, 0)$. Then $\hat{S} \subset B_{\sqrt{2rh}} \subset B_{2s}$, thanks to \Cref{obs:slice_sandwich}\ref{subobs:ball}. Assume without loss of generality that the origin lies in $P$, and that the smallest distance between a pair of distinct points in $P$ is $1$. Then $P \subset B_r$.
	
	\smallskip
	
	The proof strategy is this: First, we prove that $B_r$ contains many disjoint ball slices. We then consider a randomly translated and rotated copy of $B_r$, and argue that each ball slice therein is likely to contain a point of~$P$. Furthermore, a second-moment argument shows that it is likely that a positive proportion of these ball slices intersects $P$. We conclude by proving that any transversal (i.e., one point from each intersection of the ball slices and $P$) forms a convex independent set.
	
	\begin{proposition} \label{prop:slice_packing}
		Suppose $d \ge 2$. In $B_r$ there exist $cn^{\frac{d-1}{d+1}}$ disjoint ball slices, where $c = c(d) > 0$.
	\end{proposition}
	
	\begin{proof}
		Let $\delta > 0$ be a small number to be chosen later. The usual packing argument shows that there exist $\Omega_d(\delta^{1-d})$ points on the unit sphere in $\R^d$ with pairwise distance at least $\sqrt{2}\delta$. Indeed, we may pick the points greedily. At each step we place a small closed ball of radius $2\sqrt{2} \delta$ around every extant point. As long as the union of these balls does not cover the unit sphere, we may add another point. Since each small ball covers $O_d(\delta^{d-1})$ surface area, the greedy process yields at least $\Omega_d(\delta^{1-d})$ points. 
		
		Set $\delta\eqdef \sqrt{h/r} = n^{-\frac{1}{d+1}}$.	By applying homothety $v \mapsto rv$ to this $\sqrt{2}\delta$-separated subset of the unit sphere, we obtain a set of $\Omega_d(\delta^{1-d}) = \Omega_d \bigl( n^{\frac{d-1}{d+1}} \bigr)$ points on $\partial B_r$ with pairwise distance at least $\sqrt{2rh}$. Finally, \Cref{obs:slice_sandwich}\ref{subobs:ball} tells us that the ball slices with their apexes at these points are disjoint.
	\end{proof}
	
	We shall consider a random congruent copy of $\hat{S}$. Specifically,
	for a vector $z$ and a rotation $\rho$ we consider $z+\rho \hat{S}$;
	we sample $z$ uniformly from the open ball $B_{4r}$, and $\rho$ uniformly
	from the special orthogonal group $\SO(d)$ endowed with the usual Haar measure.
	
	\begin{lemma}\label{lem:measurebound}
		There exists $c' = c'(d, \alpha) > 0$ such that $\p \bigl((z+\rho \hat{S}) \cap P \neq \varnothing \bigr) \ge c'$.
	\end{lemma}
	
	\begin{proof}[Proof of \Cref{thm:lowerbound} assuming \Cref{lem:measurebound}]
		
		Pick two points $z \in B_{4r}, \, z' \in B_{3r}$ and a rotation $\rho \in \SO(d)$ uniformly at random. Since $\rho S_v$ and $\rho S_{re_1}$ are identically distributed,
		\begin{equation} \label{eq:trans_rot_1}
			\p\bigl((z'+\rho S_v) \cap P \neq \varnothing\bigr) = \p\bigl((z'+\rho S_{re_1}) \cap P \neq \varnothing\bigr) = \p\bigl((z'+r\rho e_1+ \rho \hat{S}) \cap P \neq \varnothing\bigr). 
		\end{equation}
		The random vector  $z'+r\rho e_1$ is uniform on a translate of $B_{3r}$ by a vector of length $r$, and that is contained in $B_{4r}$; so, it follows that
		\begin{equation} \label{eq:trans_rot_2}
			\p\bigl((z'+r\rho e_1+ \rho \hat{S}) \cap P \neq \varnothing\bigr) =  \p\bigl((z+\rho \hat{S}) \cap P \neq \varnothing \mid z\in B_{3r}+r\rho e_1\bigr).
		\end{equation}
		Because $(z+\rho \hat{S})\cap P=\varnothing$ holds unless $z\in B_{2r}$ and $B_{2r}\subset B_{3r}+r\rho e_1$, we may estimate the above as
		\begin{equation} \label{eq:trans_rot_3}
			\p\bigl((z+\rho \hat{S}) \cap P \neq \varnothing \mid z\in B_{3r}+r\rho e_1\bigr)  \geq \p\bigl((z+\rho \hat{S}) \cap P \neq \varnothing\bigr) \ge c'.
		\end{equation}
		By combining \eqref{eq:trans_rot_1}, \eqref{eq:trans_rot_2}, and \eqref{eq:trans_rot_3}, we conclude that $\p\bigl((z'+\rho S_v) \cap P \neq \varnothing\bigr) \ge c'$.
		
		Suppose $S_{v_1}, \dotsc, S_{v_m}$ are disjoint ball slices in $B_r$, where $m = cn^{\frac{d-1}{d+1}}$ as in \Cref{prop:slice_packing}. Let $N$ be the number of sets $z'+\rho S_{v_i}$ that intersect $P$. Then, by the linearity of expectation, 
		\[
		\E(N) \ge m \cdot \p\bigl((z'+\rho S_v) \cap P \neq \varnothing\bigr) \ge  cc' \cdot n^{\frac{d-1}{d+1}}.
		\]
		It follows that there exists a pair $(z_0, \rho_0) \in B_{3r} \times \SO(d)$ such that at least $N' \eqdef cc' \cdot n^{\frac{d-1}{d+1}}$ sets among $Q_i \eqdef (z_0+\rho_0 S_{v_i}) \cap P \, (i = 1, \dotsc, m)$ are nonempty. We may assume without loss of generality that $Q_1, \dotsc, Q_{N'} \neq \varnothing$. Pick $q_i \in Q_i$ arbitrarily. Then $Q \eqdef \{q_1, \dotsc, q_{N'}\}$ is a convex independent subset of $P$. Indeed, the hyperplane $\{ x \in \R^d : \langle x-z_0, v_i \rangle = r(r-h) \}$ through the base of the translated ball slice $z_0+\rho_0 S_{v_i}$  separates the point $p_i$ from the rest of $Q$. The proof is done upon setting $\beta \eqdef  cc'$.
	\end{proof}
	
	The following fact from probability theory will be useful in our proof of \Cref{lem:measurebound}. 
	
	\begin{observation}[Folklore] \label{obs:prob_CS}
		If $X$ is a random variable taking nonnegative integer values, then
		\[
		\p(X > 0) \ge \frac{(\E X)^2}{\E(X^2)}.
		\]
	\end{observation}
	
	\begin{proof}
		For any positive integer $k$, write $p_k \eqdef \p(X = k)$. The Cauchy--Schwarz inequality implies that
		\[
		\p(X > 0) \cdot \E(X^2) = \biggl( \sum_{k=1}^{\infty} p_k \biggr) \cdot \biggl( \sum_{k=1}^{\infty} k^2p_k \biggr) \ge \biggl( \sum_{k=1}^{\infty} kp_k \biggr)^2 = (\E X)^2. \qedhere
		\]
	\end{proof}
	
	\begin{proof}[Proof of \Cref{lem:measurebound}]
		Denote by $\vol(\bullet)$ the standard $d$-dimensional Lebesgue measure. For instance, we have $\vol(B_r) = \Theta_d(r^d)$. Since $\hat{S}$ is sandwiched between a cone and a cylinder of the same height $h$ and the same base area $\Theta_d(s^{d-1})$ (see \Cref{obs:slice_sandwich}), it follows that $\vol(\hat{S}) = \Theta_d(hs^{d-1}) = \Theta_{d, \alpha}(1)$.
		
		For any point $x \in P$ and any fixed rotation $\rho_0 \in \SO(d)$, observe that
		\begin{equation} \label{eq:singletondup}
			\p(x \in z+\rho_0\hat{S}) = \p(z \in x-\rho_0\hat{S}) \overset{(*)}{=} \frac{\vol(x-\rho_0\hat{S})}{\vol(B_{4r})} = \frac{\vol(\hat{S})}{\vol(B_{4r})} = \frac{\Theta_{d, \alpha}(1)}{\Theta_d \bigl( (4r)^d \bigr)} = \Theta_{d, \alpha} \Bigl( \frac{1}{n} \Bigr),
		\end{equation}
		where we used that $x - \rho_0\hat{S} \subseteq B_{4r}$ at the step marked with $(*)$.
		
		Write the random variable $X \eqdef |P \cap (z+\rho\hat{S})|$ as $X = \sum\limits_{x \in P} \mathds{1}_{\{x \in z + \rho\hat{S}\}}$, where
		\[
		\mathds{1}_{\{x \in z + \rho\hat{S}\}} \eqdef \begin{cases}
			1 \qquad &\text{when $x \in z + \rho \hat{S}$}, \\
			0 \qquad &\text{when $x \notin z + \rho \hat{S}$}.
		\end{cases}
		\]
		From the linearity of expection it follows that
		\begin{align} \label{eq:secondmom_expansion}
			\E(X^2) &= \E\Biggl( \biggl(\sum_{x\in P} \mathds{1}_{\{x \in z + \rho\hat{S}\}} \biggl)^2 \Biggr) = \E \biggl( \sum_{x\in P} \sum_{y\in P} \mathds{1}_{\{x \in z + \rho\hat{S}\}} \cdot \mathds{1}_{\{y \in z + \rho\hat{S}\}} \biggr) \nonumber \\
			&= \sum_{x \in P} \sum_{y \in P} \E \bigl( \mathds{1}_{\{x \in z + \rho\hat{S}\}} \cdot \mathds{1}_{\{y \in z + \rho\hat{S}\}} \bigr) = \sum_{x \in P} \sum_{y \in P} \p(x, y \in z+\rho\hat{S}). 
		\end{align}
		Recall that $s = \bigl( 1 - o(1) \bigr) \cdot \sqrt{2rh}$ and $\hat{S} \subset B_{2s}$. So, $\p(x, y \in z + \rho\hat{S}) = 0$ if $|x - y| \ge 4s$, and hence
		\begin{equation} \label{eq:expansion_rearrange}
			\sum_{x \in P} \sum_{y \in P} \p(x, y \in z+\rho\hat{S}) = \sum_{x \in P} \Biggl( \p(x \in z+\rho\hat{S}) + \sum_{\substack{y \in P \setminus \{x\} \\ |x-y|<4s}} \p(x, y \in z+\rho\hat{S}) \Biggr). 
		\end{equation}
		By combining \eqref{eq:secondmom_expansion} and \eqref{eq:expansion_rearrange}, from \eqref{eq:singletondup} we deduce that
		\begin{equation} \label{eq:secondmom_doubleton}
			\E(X^2) = \Theta_{d, \alpha}(1) + \sum_{x \in P} \sum_{\substack{y \in P \setminus \{x\} \\ |x-y|<4s}} \p(x, y \in z+\rho\hat{S}). 
		\end{equation}
		
		\begin{claim} \label{claim}
			\hypertarget{claim}{Let} $x, y \in B_r$ be distinct with $\abs{x - y} < 4s$. We have
			\begin{enumerate}[label=(\roman*), ref=(\roman*)]
				\item \label{pair:dependence} $\p(x, y \in z + \rho \hat{S})$ depends only on $\abs{x-y}$, and
				\item \label{pair:estimate} if $\abs{x-y} \geq h$, then $\p(x, y \in z + \rho \hat{S}) = O_{d, \alpha} \bigl( \frac{h}{|x - y|n} \bigr)$. 
			\end{enumerate}
		\end{claim}
		
		\begin{proof}
			Suppose $x,y$ and $x',y'$ are two pairs of points in $B_r$ satisfying $\abs{x-y}=\abs{x'-y'}$. Define sets
			\begin{align*}
				\Omega &\eqdef \bigl\{ (z, \rho) \in \R^d \times \SO(d) : x, y \in z + \rho\hat{S} \bigr\}, \\
				\Omega' &\eqdef \bigl\{ (z, \rho) \in \R^d \times \SO(d) : x', y' \in z + \rho\hat{S} \bigr\}. 
			\end{align*}
			Note that, because $\hat{S}\subset B_r$, the sets $\Omega$ and $\Omega'$ are both contained in $B_{4r}\times \SO(d)$. Choose $\rho_0\in\SO(d)$ and $w\in \R^d$ so that $x'=w+\rho_0 x$ and $y'=w+\rho_0 y$. Observe that the map $(z,\rho)\mapsto (w+\rho_0z,\rho_0\rho)$ is an invertible measure-preserving map from $\Omega$ to $\Omega'$; phrased differently, $\Omega$ and $\Omega'$ are translates of one another in the semidirect product $\R^d\rtimes \SO(d)$. In particular, $\Omega$ and $\Omega'$ have the same measure, which in view of $\Omega,\Omega'\subset B_{4r}\times \SO(d)$ shows $\p(x, y \in z + \rho \hat{S})=\p(x', y' \in z + \rho \hat{S})$, i.e., part \ref{pair:dependence} holds.
			
			\smallskip
			
			In view of part \ref{pair:dependence}, the probability does not change if instead of viewing $y$ as fixed, we think of it as being chosen uniformly at random from the sphere of radius $r\eqdef \abs{x-y}$ around $x$. Consider the set $\overline{S} \eqdef \{p \in \R^d : -h \le \langle p, e_1 \rangle \le 0\}$ which is a slab of width $h$ containing $\hat{S}$. Then
			\begin{align*}
				\p(x, y \in z + \rho \hat{S}) &= \p(x \in z + \rho \hat{S}) \cdot \p(y \in z+\rho \hat{S} \mid x \in z + \rho \hat{S}) \\
				&\le \p(x \in z + \rho \hat{S}) \cdot \p(y \in z+\rho \overline{S} \mid x \in z + \rho \hat{S}) \\
				&= \Theta_{d, \alpha} \Bigl( \frac{1}{n} \Bigr) \cdot \p(y \in z+\rho \overline{S} \mid x \in z + \rho \hat{S}). 
			\end{align*}
			Notice that the point $y$ is uniform on the sphere of radius $r$ around~$x$, and the slab $z+\rho \overline{S}$ intersects the sphere in a fraction $\Theta(h/r)$ of its surface. Hence \ref{pair:estimate} follows. 
		\end{proof}
		
		Write $A \lesssim B$ in place of $A = O_{d, \alpha}(B)$ for brevity. Since points of $P$ are $1$-separated and $h \to 0$, the condition $|x-y| \geq h$ holds for all pairs of distinct $x, y \in P$. It then follows from \hyperlink{claim}{Claim} that
		\begin{align} \label{eq:doubleton_estimate}
			\sum_{x \in P} \sum_{\substack{y \in P \setminus \{x\} \\ |x-y|<4s}} \p(x, y \in z+\rho\hat{S}) &\lesssim \sum_{x \in P} \sum_{\substack{y \in P \setminus \{x\} \\ |x-y|<4s}} \frac{h}{|x-y|n} \nonumber \\
			&= \frac{h}{n} \cdot \sum_{x \in P} \sum_{\substack{y \in P \setminus \{x\} \\ |x-y|<4s}} \int_{|x-y|}^{+\infty} \frac{dt}{t^2} \nonumber \\
			&\overset{(\spadesuit)}{=} \frac{h}{n} \cdot \sum_{x \in P} \sum_{\substack{y \in P \setminus \{x\} \\ |x-y|<4s}} \int_0^{+\infty} \frac{\mathds{1}_{\{t > |x-y|\}} (t)}{t^2} \, dt \nonumber \\
			&= \frac{h}{n} \cdot \sum_{x \in P} \int_0^{+\infty} \frac{\sum\limits_{\substack{\scriptscriptstyle y \in P \setminus \{x\} \\ \scriptscriptstyle |x-y|<4s}} \mathds{1}_{\{t > |x-y|\}} (t)}{t^2} \, dt \nonumber \\
			&= \frac{h}{n} \cdot \sum_{x \in P} \int_0^{+\infty} \frac{\bigl|(x+B_{\min\{t,4s\}}) \cap (P \setminus \{x\})\bigr|}{t^2} \, dt \nonumber \\
			&\overset{(\clubsuit)}{\lesssim} \frac{h}{n} \cdot \sum_{x \in P} \int_0^{+\infty} \frac{(\min\{t, 4s\})^d}{t^2} \, dt \nonumber \\
			&= h \cdot \biggl( \int_0^{4s} t^{d-2} \, dt + (4s)^{d} \int_{4s}^{+\infty} \frac{dt}{t^2} \biggr) \nonumber \\
			&\lesssim hs^{d-1} \lesssim 1. 
		\end{align}
		Here the steps marked with $(\spadesuit)$ and $(\clubsuit)$ may require some further explanations: 
		\begin{itemize}
			\item At $(\spadesuit)$, we denoted by $\mathds{1}_{\{t > |x-y|\}}$ the indicator function on $t$ of the event $t > |x-y|$. That is, 
			\[
			\mathds{1}_{\{t > |x-y|\}} (t) \eqdef \begin{cases}
				1 \qquad &\text{if $t > |x-y|$}, \\
				0 \qquad &\text{if $t \le |x-y|$}. 
			\end{cases}
			\]
			\item At $(\clubsuit)$, we use the fact that any ball centered at some $x \in P$ of radius $R$ intersects $P \setminus \{x\}$ in at most $O_d(R^d)$ points, which follows from the assumption $\min\{|a-b| : a, b \in P, \, a \neq b\} = 1$. Similar to the deduction of $\diam(P) = \Omega(\sqrt[d]{n})$ in \Cref{sec:intro}, this is seen by placing an open ball of radius $\frac{1}{2}$ centered at each point $p \in P \setminus \{x\}$ and noticing that the balls are pairwise disjoint. 
		\end{itemize}
		By combining \eqref{eq:secondmom_doubleton} and \eqref{eq:doubleton_estimate} we obtain $\E(X^2) = O_{d, \alpha}(1)$. Thus, from \Cref{obs:prob_CS} we deduce that
		\[
		\p\bigl((z+\rho \hat{S}) \cap P \neq \varnothing\bigr) = \p(X > 0) \ge \frac{(\E X)^2}{\E(X^2)} = \frac{\Theta_{d, \alpha}(1)}{O_{d, \alpha}(1)} = \Omega_{d, \alpha}(1). \qedhere
		\]
	\end{proof}
	
	\section{Convex cups and caps in higher dimensions} \label{sec:cupcap}
	In this section we describe the basic properties of convex cups and caps in $\R^d$, for general $d$.
	
	\paragraph{Definitions and notations.} For any $x = (x_1, \dotsc, x_{d-1}, x_d) \in \R^d$, define the height $h(x) \eqdef x_d$ and the projection $\pi(x) \eqdef (x_1, \dotsc, x_{d-1})$. For each $x^* \in \R^{d-1}$, the fiber $\ipi(x^*) \eqdef \{x \in \R^d : \pi(x) = x^*\}$ is the line in $\R^d$ through $x^*$ that is parallel to the $d$-th coordinate axis. If $P \subset \R^d$ is a point set and $f \colon X \to Y$ is a function with $P \subseteq X \subseteq \R^d$, then we denote $f(P) \eqdef \{f(p) : p \in P\}$.
	
	For any finite point set $P = \{p_1, \dotsc, p_m\} \subset \R^d$, we define its \emph{affine span} as 
	\[
	\spa(P) \eqdef \Biggl\{ \sum_{i=1}^m c_i p_i : \sum_{i=1}^m c_i = 1, \, c_1, \dotsc, c_m \in \R \Biggr\}. 
	\]
	Denote by $\conv(P)$ its convex hull, and by $\conv_0(P)$ the interior of $\conv(P)$. Write
	\[
	\partial_{\pi} P \eqdef \bigl\{p \in P : \pi(p) \notin \conv_0 \bigl(\pi(P)\bigr)\bigr\}. 
	\]

	For a point set $P \subset \R^d$, recall that $P$ is \emph{in general position} if no $d+1$ points of $P$ are coplanar (i.e., no $(d+1)$-subset of $P$ is contained in a $(d-1)$-hyperplane). A pair $(S, T)$ of disjoint nonempty subsets of $P$ is a \emph{generic pair} if $|S|+|T| = d+2$. Call $P$ \emph{generic} if $\lvert \spa(S) \cap \spa(T) \rvert = 1$ for every generic pair $(S, T)$ of $P$. As an illustration, suppose $P \subset \R^2$ is a planar point set. Then
	\vspace{-0.5em}
	\begin{itemize}
		\item $P$ is in general position if no three points of $P$ are coplanar, and 
		\vspace{-0.5em}
		\item $P$ is generic if no two lines spanned by points in $P$ are parallel. 
	\end{itemize}
	\vspace{-0.5em}
	We refer to a finite set $P \subset \R^d$ as \emph{regular} if
	\vspace{-0.5em}
	\begin{enumerate}[label=(R\arabic*), ref=(R\arabic*)]
		\item \label{regular:axis_inj_dup} the projection of $P$ to any coordinate axis is injective, and
		\vspace{-0.5em}
		\item \label{regular:P_gen_dup} the original set $P \subset \R^d$ is in general position and generic, and
		\vspace{-0.5em}
		\item \label{regular:pi_gen_dup} the projection set $\pi(P) \subset \R^{d-1}$ is in general position and generic. 
	\end{enumerate}
	\vspace{-0.5em}
	
	Let $P \subset \R^d$ be a regular point set and suppose $(S, T)$ is a pair of disjoint nonempty subsets of $P$ with $|S| + |T| = d+1$. Then \ref{regular:axis_inj_dup} implies that $\bigl( \pi(S), \pi(T) \bigr)$ is a generic pair of $\pi(P)$, and so from \ref{regular:pi_gen_dup} we deduce that $\pi(S)$ intersects $\pi(T)$ at a single point, say $q$. Write
	\[
	\{q_S\} \eqdef \spa(S) \cap \pi^{-1}(q), \qquad \{q_T\} \eqdef \spa(T) \cap \pi^{-1}(q). 
	\]
	Here $q_S, q_T$ are unique because \ref{regular:axis_inj_dup} implies that neither of $\spa(S), \, \spa(T)$ is parallel to the $d$-th coordinate axis. Moreover, the following algebraic property will be useful: 
	
	\begin{proposition} \label{prop:regular_alg}
		If $S = \{u_1, \dots, u_s\}$ and $T = \{v_1, \dots, v_t\}$ (hence $s+t = d+1$), then there exists a unique tuple of nonzero coefficients $(\alpha_1, \dots, \alpha_s, \beta_1, \dots, \beta_t)$ with $\sum_{i=1}^s \alpha_i = 1, \, \sum_{i=1}^t \beta_i = 1$ such that
		\[
		q_S = \sum_{i=1}^s \alpha_i u_i, \qquad q_T = \sum_{i=1}^t \beta_i v_i.  
		\]
	\end{proposition}
	
	\begin{proof}
		By symmetry, it suffices to prove the statements concerning the $\alpha$'s.
		The existence of the $\alpha$'s follows from $q_S\in \spa(S)$.
		Had one of the $\alpha$'s vanished, say $\alpha_1=0$, it would have meant that
		$\pi(S\setminus \{u_1\})\cup \pi(T)$ is not in general position in $\R^{d-1}$, contradicting \ref{regular:pi_gen_dup}. Finally, the uniqueness of $\alpha$'s follows from the set $S$ being in general position.
	\end{proof}
	
	With $S$ and $T$ being defined as above, we say that $S$ \emph{lies above} $T$ if $h(q_S) > h(q_T)$. For disjoint finite point sets $A, B \subset \R^d$, we say that $A$ is \emph{high above} $B$ if $A \cup B$ is regular and $S$ lies above $T$ for every pair $(S, T)$ of $A \cup B$ with $S \subseteq A, \, T \subseteq B$ such that $\bigl(\pi(S), \pi(T)\bigr)$ is a generic pair of $\pi(A \cup B)$ (i.e., $|S| + |T| = d+1$). The letters $A$ and $B$ are chosen as abbreviations for ``above'' and ``below''. 
	
	\smallskip
	
	We are ready to generalize the definitions of convex cups and caps (\cite{valtr92}) to higher dimensions. Let $P \subset \R^d$ be a regular point set. Call $P$ a \emph{convex cup} if for any $p, p_1, \dotsc, p_d \in P$, 
	\[
	\pi(p) = \sum_{i=1}^d c_i \pi(p_i) \qquad \text{implies} \qquad h(p) < \sum_{i=1}^d c_i h(p_i), 
	\]
	where $c_1, \dotsc, c_d \in (0, 1)$ and $c_1 + \dots + c_d = 1$. Call $P$ a \emph{convex cap} if for any $p, p_1, \dotsc, p_d \in P$,  
	\[
	\pi(p) = \sum_{i=1}^d c_i \pi(p_i) \qquad \text{implies} \qquad h(p) > \sum_{i=1}^d c_i h(p_i), 
	\]
	where $c_1, \dotsc, c_d \in (0, 1)$ and $c_1 + \dots + c_d = 1$. These definitions generalize those of planar convex cups and caps in \cite{valtr92}. Moreover, such cups and caps are crucial to our proof of \Cref{thm:upperbound}. 
	
	\paragraph{Properties.} We record a decomposition result which is crucial to the estimates in \Cref{sec:osci}. 
	
	\begin{proposition} \label{prop:cupcap_decomp}
		For every regular convex independent set $C \subset \R^d$, there exist a convex cap $C_A\subseteq C$ and a convex cup $C_B\subseteq C$ such that $C_A \cup C_B = C$ and $C_A \cap C_B = \partial_{\pi} C$.
	\end{proposition}
	
	\begin{proof}
		If $|C| \le d$, then $\pi(C) \subset \R^{d-1}$ is convex independent by \ref{regular:pi_gen_dup}, since every $d$ points in $\R^{d-1}$ that are in general position serve as the vertices of a convex polytope. It then follows that $C = \partial_{\pi} C$ is both a convex cup and a convex cap, and hence $C_A = C_B = \partial_{\pi} C$ gives the required decomposition. 
		
		Suppose $|C| \ge d+1$ then, which implies $|\partial_{\pi} C| \ge d$. For any point $p \in C \setminus \partial_{\pi} C$, the fiber 
		\[
		\ell_p \eqdef \ipi\bigl( \pi(p) \bigr) = \{x \in \R^d : \pi(x) = \pi(p)\} 
		\]
		intersects $\conv(\partial_{\pi} C)$ in a line segment $s_p$. If $|\partial_{\pi} C| = d$, then \ref{regular:pi_gen_dup} implies that $\conv(\partial_{\pi} C)$ spans a $(d-1)$-dimensional hyperplane in $\R^d$, and hence the segment $s_p$ degenerates to a single point (but we still call it a segment). Since $p \notin \partial_{\pi} C$ and $C$ is convex, $p$ is either above $s_p$ or below $s_p$. Put $p$ into $C_A$ if $p$ is above $s_p$, and into $C_B$ if $p$ is below $s_p$. Add all points of $\partial_{\pi}C$ into both $C_A$ and $C_B$. 
		
		We claim that the pair $(C_A, C_B)$ fulfills the decomposition conditions. By symmetry, it suffices to show that $C_A$ is a convex cap. Assume to the contrary that there are points $p', p_1, \dotsc, p_d \in C_A$ and coefficients $\alpha_1, \dotsc, \alpha_d \in (0, 1)$ with $\alpha_1 + \dots + \alpha_d = 1$ such that 
		\[
		\pi(p') = \sum_{i=1}^d \alpha_i \pi(p_i) \qquad \text{and} \qquad h(p') < \sum_{i=1}^d \alpha_i h(p_i). 
		\]
		The inequality is strict because $C$ is in general position. The first condition implies that $p' \in C \setminus \partial_{\pi} C$. Denote by $p_A \eqdef \sum_{i=1}^d \alpha_i p_i$ the intersection point of the line $\ell_{p'}$ with $\spa(\{p_1, \dotsc, p_d\})$. The second condition implies that $p_A$ is above $p'$. Denote by $p_B$ the upper endpoint of the segment $s_{p'}$ which is below the point $p'$. Then it follows from $p_B \in \conv(\partial_{\pi} C)$ that
		\[
		p' \in \conv_0(\{p_A, p_B\}) \subseteq \conv_0(\{p_1, \dotsc, p_d\} \cup \partial_{\pi} C) \subseteq \conv_0(C), 
		\]
		which contradicts the assumption that $C$ is convex independent. The proof is complete. 
	\end{proof}
	
	\begin{proposition} \label{thm:cupcap_intpoint}
		Suppose $A, B \subset \R^d$ are point sets such that $A$ is high above $B$. If $P \subseteq A \cup B$ is a convex cup or cap, then either $P \setminus \partial_{\pi} P \subseteq A$ or $P \setminus \partial_{\pi} P \subseteq B$. 
	\end{proposition}
	
	This result, which will be crucial in the analysis of our construction, relies on the following lemma.
	
	\begin{lemma} \label{lem:cupcap_intpoint}
		Suppose $A, B \subset \R^d$ are two point sets such that $A$ is high above $B$. Let $\{p\} \cup K$ be a $(d+1)$-element subset of $A \cup B$ with $p \notin K$ and $\pi(p) \in \conv_0 \bigl(\pi(K)\bigr)$. 
		\begin{enumerate}[label=(\roman*), ref=(\roman*)]
			\item If $\{p\} \cup K$ is a convex cup with $p \in A$, then $K \subseteq A$. 
			\item If $\{p\} \cup K$ is a convex cap with $p \in B$, then $K \subseteq B$. 
		\end{enumerate}
	\end{lemma}
	
	\begin{proof}
		Let $K \eqdef \{p_1, \dotsc, p_d\}$. Suppose $\{p\} \cup K$ is a convex cup with $p \in A$. The other case is similar. 
		
		Suppose to the contrary that $K \cap B \neq \varnothing$. Assume $p_1, \dotsc, p_k \in B$ and $p_{k+1}, \dotsc, p_d \in A$, where $1 \le k \le d$. Consider the pair $(S, T) \eqdef \bigl(\{p_{k+1}, \dotsc, p_d, p\}, \{p_1, \dotsc, p_k\}\bigr)$ of $A \cup B$. By \ref{regular:pi_gen_dup} we obtain that $\pi \bigl(\spa(S)\bigr)$ and $\pi \bigl(\spa(T)\bigr)$ intersect at a unique point $q^* \in \R^{d-1}$. Set
		\[
		\{q_T\} \eqdef \spa(T) \cap \ipi(q^*), \qquad
		\{q_S\} \eqdef \spa(S) \cap \ipi(q^*). 
		\]
		It follows from \Cref{prop:regular_alg} that there are unique nonzero coefficients $\alpha_1, \dotsc, \alpha_k$ with $\sum_{i=1}^k \alpha_i = 1$ and $\beta, \beta_{k+1}, \dotsc, \beta_d$ with $\beta + \sum_{i=k+1}^d \beta_i = 1$ such that 
		\begin{equation} \label{eq:combin_aff_dup}
			q_T = \sum_{i=1}^k \alpha_i p_i, \qquad q_S = \beta p + \sum_{i=k+1}^d \beta_i p_i. 
		\end{equation}
		
		Since $A$ is high above $B$, it follows that $h(q_S) > h(q_T)$, and hence $q_S = q_T + \lambda e_d$ for some $\lambda > 0$, where $e_d \eqdef (0, \dotsc, 0, 1)$. This together with \eqref{eq:combin_aff_dup} show that 
		\begin{equation} \label{eq:combin_p_dup}
			\beta p = \sum_{i=1}^k \alpha_i p_i - \sum_{i=k+1}^d \beta_i p_i + \lambda e_d. 
		\end{equation}
		
		By applying $\pi$ to the both sides of~\eqref{eq:combin_p_dup}, we deduce that 
		\begin{equation} \label{eq:combin_proj_dup}
			\pi(p) = \sum_{i=1}^k (\beta^{-1}\alpha_i) \pi(p_i) + \sum_{i=k+1}^d (-\beta^{-1}\beta_i) \pi(p_i) = \sum_{i=1}^d c_i \pi(p_i), 
		\end{equation}
		where $c_i \eqdef \beta^{-1} \alpha_i$ for $i = 1, \dotsc, k$ and $c_i \eqdef -\beta^{-1} \beta_i$ for $i = k+1, \dotsc, d$. Observe that
		\[
		\sum_{i=1}^n c_i = \sum_{i=1}^k \beta^{-1} \alpha_i - \sum_{i=k+1}^d \beta^{-1} \beta_i = \beta^{-1} - \beta^{-1} (1 - \beta) = 1. 
		\]
		Since $\pi(p) \in \conv_0\bigl(\pi(K)\bigr)$ and the expression of $\pi(p)$ as a convex combination of $\pi(K)$ is unique, it follows from \eqref{eq:combin_proj_dup} that $c_i\in (0,1)$ for all~$i$. If $\beta < 0$, then $\alpha_i < 0$ for $i = 1, \dotsc, k$, which contradicts $\sum_{i=1}^k \alpha_i = 1$. So, $\beta > 0$. By applying $h$ to the both sides of \eqref{eq:combin_p_dup} we obtain
		\[
		h(p) = \sum_{i=1}^k (\beta^{-1}\alpha_i) h(p_i) + \sum_{i=k+1}^d (-\beta^{-1}\beta_i) h(p_i) + \beta^{-1} \lambda = \sum_{i=1}^d c_i h(p_i) + \beta^{-1} \lambda > \sum_{i=1}^d c_i h(p_i), 
		\]
		which contradicts $\{p\} \cup K$ being a convex cup. We conclude that $K \subseteq A$. 
	\end{proof}
	
	\begin{proof}[Proof of \Cref{thm:cupcap_intpoint}] 
		Suppose $P$ is a convex cup. The other case is similar. If $(P \setminus \partial_{\pi} P) \cap A = \varnothing$, then $P \setminus \partial_{\pi} P \subseteq B$, and the proof is done. We may assume that $p \in (P \setminus \partial_{\pi} P) \cap A$ then. 
		
		From $P \setminus \partial_{\pi} P \neq \varnothing$ we see that $|P| \ge d+1$, and so $\conv\bigl(\pi(\partial_{\pi} P)\bigr)$ is a $(d-1)$-dimensional convex polytope. Fix any $p' \in P \setminus (\partial_{\pi} P \cup \{p\})$. The ray $\overrightarrow{\pi(p')\pi(p)}$ intersects a unique facet of the polytope $\conv\bigl(\pi(\partial_{\pi} P)\bigr)$, say $F$. Since \ref{regular:pi_gen_dup} implies that every facet of $\conv\bigl(\pi(\partial_{\pi} P)\bigr)$ consists of exactly $d-1$ points, there exist distinct $p_1, \dotsc, p_{d-1} \in \partial_{\pi} P$ such that $F = \conv\bigl(\pi(\{p_1, \dotsc, p_{d-1}\})\bigr)$. We claim that
		\[
		\pi(p) \in \conv_0 \bigl(\pi(\{p', p_1, \dotsc, p_{d-1}\})\bigr). 
		\]
		Indeed, this follows from \Cref{prop:regular_alg} applied to the pair $(S, T) \eqdef \bigl( \{p, p'\}, \{p_1, \dots, p_{d-1}\} \bigr)$. Finally, since $p \in A$, \Cref{lem:cupcap_intpoint} with $\{p', p_1, \dots, p_{d-1}\}$ in the place of $K$ implies that $\{p', p_1, \dots, p_{d-1}\} \subseteq A$. In particular, $p' \in A$. Since $p' \in (\partial_{\pi} P \cup \{p\})$ is arbitrary, we conclude that $P \setminus \partial_{\pi} P \subseteq A$. 
	\end{proof}
	
	\section{Oscillators} \label{sec:osci}
	
	This section is devoted to the definition and properties of \emph{oscillators}. 
	
	\paragraph{Definitions and basic properties.} For any $x = (x_1, \dotsc, x_d) \in \R^d$, we write $\pi_1(x) \eqdef x_1$. 
	
	A sequence of points $(p_1, p_2, \dotsc)$ in $\R^d$ is called \emph{ordered} if $\pi_1(p_1) < \pi_1(p_2) < \dotsb$. By \ref{regular:axis_inj_dup}, any regular point set $P \subset \R^d$ can be sorted to make an ordered sequence. We shall abuse the notations by writing $P = \{p_1, \dotsc, p_m\}$ for the sorted ordered sequence formed by some regular point set $P$. That is, $\pi_1(p_1) < \dotsb < \pi_1(p_m)$. We also write $P_{a, b} \eqdef \{p_t \in P : t \equiv a \pmod b\}$ for $b \ge 2$ and $a \in [b]$. For instance, $P_{1, 2} = \{p_1, p_3, \dotsc\}$ and $P_{2, 2} = \{p_2, p_4, \dotsc\}$. 
	
	We define \emph{$d$-oscillators} recursively. Let $P = \{p_1, \dotsc, p_m\} \subset \R^d$ be a finite point set. 
	\vspace{-0.5em}
	\begin{enumerate}[label=(O\arabic*), ref=(O\arabic*)]
		\item \label{osci:dim1} If $m = 1$ or $d = 1$, then $P$ is a $d$-oscillator. 
		\vspace{-0.5em}
		\item \label{osci:dim2} Suppose $m \ge 2$ and $d \ge 2$. If
		\vspace{-0.5em}
		\begin{itemize}
			\item[--] $P$ is regular, and
			\item[--] both $P_{1, 2}$ and $P_{2, 2}$ are $d$-oscillators, and
			\item[--] either $P_{1, 2}$ is high above $P_{2, 2}$ or $P_{2, 2}$ is high above $P_{1, 2}$, and
			\item[--] $\pi(P) \subset \R^{d-1}$ is a $(d-1)$-oscillator, 
		\end{itemize}
		\vspace{-0.5em}
		then $P$ is a $d$-oscillator. 
	\end{enumerate}
	\vspace{-0.5em}
	
	For $d=2$ this definition coincides with the definition of Horton sets \cite{horton}. However, this definition differs from generalizations of Horton sets discussed in \cite{valtr92'} and \cite{conlon_lim}. The reason for the difference is that \cite{valtr92'} and \cite{conlon_lim} were both motivated by avoidance of large holes.
	
	\begin{proposition} \label{prop:oscillator}
		Suppose $P = \{p_1, \dotsc, p_m\} \subset \R^d$ is a $d$-oscillator. 
		\begin{enumerate}[label=(\roman*), ref=(\roman*)] 
			\item \label{osci:section}
			For any integers $a, b$ with $1 \le a \le b \le m$, the set $P_{[a, b]} \eqdef \{p_a, p_{a+1}, \dotsc, p_b\}$ is a $d$-oscillator. 
			\item \label{osci:subset} 
			For any integers $N \ge 2$ and $r \in [N]$, the set $P_{r, N}$ is a $d$-oscillator. 
		\end{enumerate}
	\end{proposition}
	
	\begin{proof}
		We prove both statements by induction on $d$. The $d=1$ base case is trivial by \ref{osci:dim1}. 
		
		If \ref{osci:section} is proved for $d-1$, then $\pi(P_{[a, b]}) = \bigl(\pi(P)\bigr)_{[a, b]}$ is a $(d-1)$-oscillator. By \ref{osci:dim2}, it suffices to show that the sets $A \eqdef (P_{[a, b]})_{1, 2}$ and $B \eqdef (P_{[a, b]})_{2, 2}$ are both $d$-oscillators and that either $A$ is high above $B$ or $B$ is high above~$A$. To see this, we induct on $|P_{[a, b]}|$. The $|P_{[a, b]}| = 1$ case is trivial by \ref{osci:dim1}, and the inductive step is done by noticing that $A$ and $B$ are consecutive sections of $P_{1, 2}$ and $P_{2, 2}$, respectively, and that $\max\{|A|, |B|\} < |P_{[a, b]}|$. 
		
		If \ref{osci:subset} is proved for $d-1$, then $\pi(P_{r, N}) = \bigl(\pi(P)\bigr)_{r, N}$ is a $(d-1)$-oscillator. By \ref{osci:dim2}, it suffices to show that $A \eqdef (P_{r, N})_{1, 2}, \, B \eqdef (P_{r, N})_{2, 2}$ are $d$-oscillators and that $A$ is high above $B$ or vice versa. 
		\vspace{-0.5em}
		\begin{itemize}
			\item If $N$ is odd, then we induct on $|P_{r, N}|$. To be specific, in the inductive hypothesis $N$ is fixed while $r \in [N]$ is arbitrary. The $|P_{r, N}| = 1$ case is trivial. 
			\vspace{-0.5em}
			\begin{itemize}
				\item If $r$ is odd, then $A = (P_{1, 2})_{\frac{r+1}{2}, N}, \, B = (P_{2, 2})_{\frac{r+N}{2}, N}$, and so $A$ is high above $B$ or $B$ is high above $A$. Since $\max\{|A|, |B|\} < |P_{r, N}|$, the inductive proof is complete. 
				\vspace{-0.25em}
				\item If $r$ is even, then $A = (P_{2, 2})_{\frac{r}{2}, N}, \, B = (P_{1, 2})_{\frac{r+N+1}{2}, N}$, and so $A$ is high above $B$ or $B$ is high above $A$. Since $\max\{|A|, |B|\} < |P_{r, N}|$, the inductive proof is complete. 
			\end{itemize}
			\vspace{-0.5em}
			\vspace{-0.5em}
			\item For even values of $N$, it suffices to prove that the statement ``\emph{$A$ and $B$ are both $d$-oscillators and either $A$ is high above $B$ or $B$ is high above $A$}'' holds for $\frac{N}{2}$ implies that the same holds for~$N$. This follows from $P_{r, N} = (P_{1, 2})_{\frac{r+1}{2}, \frac{N}{2}}$ when $r$ is odd and $P_{r, N} = (P_{2, 2})_{\frac{r}{2}, \frac{N}{2}}$ when $r$ is even. 
		\end{itemize}
		\vspace{-0.5em}
		The inductive proof is complete by combining the cases above. 
	\end{proof}
	
	For any finite point set $P \subset \R^d$, a bijection $\tau \colon P \to P'$ is called a \emph{$\delta$-perturbation} if $\|\tau(p) - p\| \le \delta$ for every $p \in P$. The following fact is obvious: 
	\begin{observation} \label{obs:perturbation}
		For any $d$-oscillator $P \subset \R^d$, there exists $\delta_0 > 0$ such that for every $\delta_0$-perturbation $\tau$ the set $\tau(P)$ is a $d$-oscillator as well. 
	\end{observation}
	
	\paragraph{Existence of oscillators.} We are to construct $d$-oscillators of arbitrarily large size recursively on dimension $d$. For any integer $N \ge 0$, write it in binary as $N = \sum_{k \ge 0} a_k \cdot 2^k \, (a_k \in \{0, 1\})$. Set 
	\[
	(N)_{\veps} \eqdef \sum_{k \ge 0} a_k \veps^{k+1}. 
	\]
	The idea is to start with the set $\bigl\{ \bigl(k, (k)_{\veps_2}, \dotsc, (k)_{\veps_d}\bigr) : k \in [m] \bigr\}$, where the constants $\veps_2, \dotsc, \veps_d$ are chosen to be $1 \gg \veps_2 \gg \dotsb \gg \veps_d > 0$, and then to perturb it properly to ensure the regularity. 
	
	A bijection $\tau \colon P \to P'$ is called a $\delta$-\emph{height-perturbation} if $\pi(\tau(p)) = \pi(p)$ and $|h(\tau(p))-h(p)| \le \delta$ for every $p \in P$. Moreover, we say that $\tau$ is a \emph{regular $\delta$-height-perturbation} if $\tau(P) = P'$ is regular. 
	
	Any $P \subset \R^1$ is regular, and any $\delta$-height-perturbation is a $\delta$-perturbation. Note that the points in the next lemma are indexed from~$0$, not~$1$. This is done to make the map $k \mapsto (k)_{\veps}$ well-behaved. 
	
	\begin{lemma} \label{lem:osci_lifting}
		Suppose $d \ge 2$. For any generic $(d-1)$-oscillator $P^* \eqdef \{p_0^*, p_1^*, \dotsc, p_{\ell}^*\} \subset \R^{d-1}$ and every sufficiently small $\veps > 0$, there exists $\delta > 0$ such that every regular $\delta$-height-perturbation of
		\[
		P \eqdef \big\{\bigl(p_k^*, (k)_{\veps}\bigr) : k \in 0, 1, \dotsc, \ell\big\}
		\]
		is a $d$-oscillator. Here $\bigl(p_k^*, (k)_{\veps}\bigr)$ is obtained by appending a $d$-th coordinate $(k)_{\veps}$ to $p_k^*$.
	\end{lemma}
	
	\begin{proof}
		Write $p_k \eqdef \bigl(p_k^*, (k)_{\veps}\bigr)$. We induct on $\ell \in \Z_{\ge 0}$. The $\ell = 0$ base case is obvious by \ref{osci:dim1}. 
		
		Assume that the statement is established for $\ell-1$, where $\ell \ge 1$. Let
		\[
		Q \eqdef \{q_0, q_1, \dotsc, q_{\ell}\}
		\]
		be a regular $\delta$-height-perturbation of $P$ with $\pi(q_k) = p_k^*$ for $k = 0, 1, \dotsc, \ell$. For sufficiently small $\veps > 0$ and $\delta > 0$, we claim that $Q_{2, 2}$ is high above $Q_{1, 2}$, and that both $Q_{2, 2}$ and $Q_{1, 2}$ are $d$-oscillators. Recall that $Q_{2, 2} = \{q_0, q_2, q_4, \dotsc\}$ and $Q_{1, 2} = \{q_1, q_3, q_5, \dotsc\}$. 
		
		We first verify that $Q_{2, 2}$ is high above $Q_{1, 2}$ by specifying $\veps$. Let $(S, T)$ be a pair of subsets of $Q$ with $S \subseteq Q_{2, 2}$ and $T \subseteq Q_{1, 2}$ such that $\bigl( \pi(S), \pi(T) \bigr)$ is a generic pair of $\pi(Q)$. Since $\pi(Q) = P^*$ is regular, hence generic by \ref{regular:pi_gen_dup}, we see that $\pi\bigl(\spa(S)\bigr) \cap \pi\bigl(\spa(T)\bigr)$ consists of a single point $q^*_{S, T}$. Moreover, it is important to notice that the point $q_{S, T}^*$ is independent on $\veps$. Define
		\[
		\{q_{S, T}^S\} \eqdef \spa(S) \cap \pi^{-1} (q^*_{S, T}), \qquad \{q_{S, T}^T\} \eqdef \spa(T) \cap \pi^{-1} (q^*_{S, T}). 
		\]
		Observe that $h(q_{S, T}^S) = \veps + O(\veps^2+\delta)$ and $h(q_{S, T}^T) = O(\veps^2+\delta)$ as $\veps \to 0^+$. Thus, if we choose $\veps > 0$ to be sufficiently small and $\delta \in (0, \veps^2)$, then $h(q_{S, T}^S) > h(q_{S, T}^T)$, and hence $Q_{2, 2}$ is high above $Q_{1, 2}$. 
		
		We then show that $Q_{1, 2}, Q_{2, 2}$ are $d$-oscillators. Define $h_k$ by $q_k - p_k= h_k e_d$ for $k = 0,1,\dotsc, \ell$. Then $|h_k| \le \delta$ for all $k$. By the inductive hypothesis, there exists some $\delta' > 0$ such that every regular $\delta'$-height-perturbation of $P^*_{1, 2}$ or $P^*_{2, 2}$ is a $d$-oscillator. Set $\ell_1 \eqdef \lceil \frac{\ell+1}{2} \rceil$, $\ell_2 \eqdef \lfloor \frac{\ell+1}{2} \rfloor$. Then 
		\begin{align*}
			\widetilde{P}_1 &\eqdef \{p_{2k}+\veps^{-1}h_{2k}e_d : k =0, 1, \dotsc, \ell_1-1\}, \\
			\widetilde{P}_2 &\eqdef \{p_{2k+1}+\veps^{-1}h_{2k+1}e_d : k = 0, 1, \dotsc, \ell_2-1\}
		\end{align*}
		are both $2$-oscillators as long as $\delta < \veps\delta'$. Consider the map $\varphi \colon (x_1, \dotsc, x_{d-1}, x_d) \mapsto (x_1, \dotsc, x_{d-1}, \veps x_d)$. Since $Q_{1, 2} = \varphi(\widetilde{P}_1)$ and $Q_{2, 2} = \varphi(\widetilde{P}_2) + \veps e_d$, it follows that $Q_{1, 2}, Q_{2, 2}$ are both $d$-oscillators. 
		
		The inductive proof is done by setting $\delta < \min\{\veps^2, \veps\delta'\}$.
	\end{proof}
	
	The natural next step is to recursively use \Cref{lem:osci_lifting} to construct oscillators in every dimension. To do so, we must show the existence of regular $\delta$-height-perturbations, which is what we do next.
	
	\begin{lemma} \label{lem:reg_perturb}
		Suppose $d \ge 2$. Let $P^* \eqdef \{p_1^*, \dotsc, p_{\ell}^*\} \subset \R^{d-1}$ be a $(d-1)$-oscillator and $h_1, \dots, h_{\ell}\in \R$ be arbitrary heights. Then for any $\delta > 0$, there exists a regular $\delta$-height-perturbation of
		\[
		P \eqdef \big\{(p_k^*, h_k) : k \in 1, \dotsc, \ell\big\}. 
		\]
	\end{lemma}
	
	\begin{proof}
		Denote by $p_i \eqdef (p_i^{*}, h_i)$ the $i$-th point of~$P$. Let $P_i \eqdef \{p_1, \dots, p_i\}$. Note that $P_{\ell}=P$. To define a regular $\delta$-height-perturbation $\tau$ of $P$, we specify what it does to $p_1, p_2, \dotsc, p_{\ell}$ in order. 
		
		Suppose $\tau(p_1), \dotsc, \tau(p_i)$ have been defined. Write $p_j' \eqdef \tau(p_j)$ for $j \leq i$ and $P_i' \eqdef \{p_1, \dotsc, p_i'\}$. 
		
		Call a point $p' \in \ipi(p_{i+1}^*)$ \emph{bad} if the set $P_i'\cup \{p'\}$ is not regular. It then suffices to show that the number of bad points is finite. Indeed, this implies that there exists $h_{i+1}' \in (h_{i+1} - \delta, h_{i+1} + \delta)$ such that $p_{i+1}' \eqdef (p_{i+1}^*, h_{i+1}')$ is not bad, and hence $P_{i+1}' \eqdef P_i' \cup \{p_{i+1}'\}$ is a regular $\delta$-height-perturbation of $P_{i+1}$. By taking $\tau(p_{i+1}) \eqdef p_{i+1}'$, this concludes the inductive construction of $\tau$.
		
		In \Cref{append:bad} we will prove that every bad point is the intersection of the vertical line $\ipi(p_{i+1}^*)$ with some hyperplane not containing any vertical line, and the number of such hyperplanes is finite. This implies the finiteness claim on the number of bad points, and hence concludes the proof. 
	\end{proof}
	
	Now we are ready to construct oscillators in any dimension, by induction on the dimension.
	
	\begin{proposition} \label{thm:oscillator_dup}
		For any positive integer $\ell$ and any $\veps > 0$, there exists a $d$-oscillator which is an $\veps$-perturbation image of the set $L \eqdef \{ke_1 : k \in [\ell]\}$, where $e_1 \eqdef (1, 0, \dotsc, 0)$. 
	\end{proposition}
	
	\begin{proof}
		For $t = 1, \dots, d$, let $\veps_t' \eqdef \frac{\veps}{2^{t+2}}$ and $L^t \subset \R^t$ be the projection of $L$ onto its first $t$ coordinates. 
		
		Start with the $1$-oscillator $L_1 \eqdef [\ell]$. By \Cref{lem:osci_lifting}, there are $\veps_1 < \veps_1'$ and $\delta_1 > 0$ such that any regular $\delta_1$-height-perturbation of $L_2^- \eqdef \bigl\{\bigl(k, (k)_{\veps_1}\bigr) : k \in [\ell]\bigr\}$ is a $2$-oscillator. Set $\delta_1' \eqdef \min\{\delta_1, \veps_1'\}$. By \Cref{lem:reg_perturb}, there exists a $2$-oscillator $L_2$ which is a regular $\delta_1'$-height-perturbation of $L_2^-$. Then $L_2$ is a $(2\veps_1' + \delta_1')$-perturbation (because $(k)_{\veps_1} < 2\veps_1$), hence a $(3\veps_1')$-perturbation of $L^2$. 
		
		For $t = 3, \dots, d$, suppose the $(t-1)$-oscillator $L_{t-1}$, a $(3\veps_1' + \dots + 3\veps_{t-2}')$-perturbation of $L^{t-1}$, has already been generated. Then through the same process as in the construction of $L_2$ we may build up a $t$-oscillator $L_t$ which is a $(3\veps_1' + \dots + 3\veps_{t-1}')$-perturbation of $L^t$. Finally, the $d$-oscillator $L_d$ is a $(3\veps_1' + \dots + 3\veps_{d-1}')$-perturbation, hence an $\veps$-perturbation of $L^d = L$, as desired. 
	\end{proof}
	
	\paragraph{Convex independent sets among oscillators.} 
	We next show that oscillators do not contain large convex independent subsets. 
	
	For any $d$-oscillator $P = \{p_1, p_2, \dotsc\} \subset \R^d \ (|P| < \infty)$ and any nonempty convex independent~subset $C = \{p_{i_1}, \dotsc, p_{i_k}\} \subset P$ with $i_1 < \dotsb < i_k$, we define $\str_P(C) \eqdef i_k - i_1$ as the \emph{stretch} of $C$ in~$P$. Recall that $\pi_1(p_1) < \pi_1(p_2) < \dotsb$. Let $\mathcal{O}_d$ be the set of all $d$-oscillators in $\R^d$. Define for any $P \in \cO_d$
	\begin{align*}
		\cC_k(P) &\eqdef \Bigl\{C \in \tbinom{P}{k} : \text{$C$ is convex independent} \Bigr\}, \\
		\cC'_k(P) &\eqdef \Bigl\{C \in \tbinom{P}{k} : \text{$C$ is a convex cup or cap} \Bigr\}, 
	\end{align*}
	where $\binom{P}{k}$ refers to the set of all $k$-element subsets of $P$. We also define
	\begin{align*}
		\Delta_d(k) &\eqdef \min_{P \in \cO_d} \min_{C \in \cC_k(P)} \str_P(C), \\
		\Delta'_d(k) &\eqdef \min_{P \in \cO_d} \min_{C \in \cC'_k(P)} \str_P(C), 
	\end{align*}
	with the convention that $\displaystyle \min_{C \in \varnothing} \str_P(C) = +\infty$. Here the reason of writing ``$\min$'' instead of ``$\inf$'' is that $\str_P(C)$ is always a non-negative integer. Since every subset of $\R^d$ containing at most $d+1$ points is trivially convex independent, we have $\Delta_d(k) = \Delta'_d(k) = k-1$ for $k = 1, \dotsc, d+1$. 
	
	The quantities $\Delta_d(k)$ and $\Delta'_d(k)$ measure the minimum stretch in $d$-oscillators. We first establish that $\Delta_d(k)$ and $\Delta'_d(k)$ are not far apart, and then use it to bound both from below.
	\begin{proposition} \label{prop:Delta_bound}
		We have $\Delta'_d(\lceil k/2 \rceil) \le \Delta_d(k) \le \Delta'_d(k)$ for all $d \ge 2$ and all $k \ge 1$. 
	\end{proposition}
	
	\begin{proof}
		The upper bound is directly implied by the inclusion $\cC'_k(P) \subseteq \cC_k(P)$. For the lower bound, assume $\Delta_d(k) = \str_P(C)$, where $P \in \cO_d$ and $C \in \cC_k(P)$. By \Cref{prop:cupcap_decomp}, there exist a convex cap $C_A$ and a convex cup $C_B$ such that $C = C_A \cup C_B$, and hence $|C_A| \ge \lceil k/2 \rceil$ or $|C_B| \ge \lceil k/2 \rceil$ (say $\abs{C_A}\ge \lceil k/2 \rceil$). It then follows that $\Delta'_d(\lceil k/2 \rceil) \le \str_P(C_A) \le \str_P(C) = \Delta_d(k)$.
	\end{proof}
	
	\begin{theorem} \label{thm:expstretch}
		For any $d \ge 2$ and any $k \ge 2$, we have that 
		\[
		\Delta'_d(k) \ge \frac{1}{2} \exp\Biggl(\frac{k^{\frac{1}{d-1}}}{4d-4}\Biggr), \qquad \Delta_d(k) \ge \frac{1}{2} \exp\Biggl(\frac{k^{\frac{1}{d-1}}}{4d}\Biggr). 
		\]
	\end{theorem}
	
	\begin{proof}
		We induct on $d$. Notice that $\str_P(C) = 2\str_{P_{t, 2}}(C)$ if $C \subseteq P_{t, 2}$ for some $t \in \{1, 2\}$. 
		
		Suppose $d = 2$. Let $P \in \cO_2$ and $C \in \cC'_k(P)$ with $k \ge 2$ being arbitrary. Since $\partial_{\pi} C \subset \R^1$ is convex independent, it follows from \ref{regular:axis_inj_dup} that $|\partial_{\pi} C| = 2$, and hence $|C \setminus \partial_{\pi} C| = k-2$. By \Cref{thm:cupcap_intpoint}, we may assume without loss of generality that $C \setminus \partial_{\pi} C \subseteq P_{1, 2}$, and hence 
		\[
		\str_P(C) \ge \str_{P}(C \setminus \partial_{\pi} C) = 2 \str_{P_{1, 2}}(C \setminus \partial_{\pi} C) \ge 2\Delta'_2(k-2). 
		\]
		This together with the facts $\Delta'_2(2) = 1, \, \Delta'_2(3) = 2$ imply that $\str_P(C) \ge \frac{1}{2} \bigl(\sqrt{2}\bigr)^k \ge \frac{1}{2} \exp\bigl(\frac{k}{4}\bigr)$. So, $\Delta'_2(k) \ge \frac{1}{2} \exp\bigl(\frac{k}{4}\bigr)$, and hence $\Delta_2(k) \ge \frac{1}{2} \exp\bigl(\frac{k}{8}\bigr)$ by \Cref{prop:Delta_bound}. The $d = 2$ case is done. 
		
		Assume both inequalities hold for $d-1 \, (d \ge 3)$ and all $k \ge 2$. Consider the $d$ case. Let $P \in \cO_d$ and $C \in \cC'_k(P)$ with $k \ge 2$ being arbitrary. Set $m \eqdef |\partial_{\pi} C|$, and hence $m \ge 2$. We claim that 
		\begin{equation} \label{eq:stretchbound}
			\str_P(C) \ge \max \bigl\{\Delta_{d-1}(m), \, 2\Delta'_d(k-m)\bigr\}. 
		\end{equation}
		On one hand, from $\pi(\partial_{\pi} C) \in \cC_m\bigl(\pi(P)\bigr)$ we deduce that 
		\begin{equation} \label{eq:stretch1}
			\str_P(C) \ge \str_P(\partial_{\pi} C) \ge \str_{\pi(P)} \bigl(\pi(\partial_{\pi} C)\bigr) \ge \Delta_{d-1}(m). 
		\end{equation}
		On the other hand, by assuming $C \setminus \partial_{\pi} C \subseteq P_{1, 2}$ without loss of generality (\Cref{thm:cupcap_intpoint}), we have 
		\begin{equation} \label{eq:stretch2}
			\str_P(C) \ge \str_P(C \setminus \partial_{\pi} C) = 2 \str_{P_{1, 2}}(C \setminus \partial_{\pi} C) \ge 2\Delta'_d(k-m). 
		\end{equation}
		Thus, our claimed estimate \eqref{eq:stretchbound} follows from \eqref{eq:stretch1} and \eqref{eq:stretch2}. 
		
		Now we are ready to prove 
		\begin{equation} \label{eq:Delta'}
			\str_P(C) \ge \frac{1}{2} \exp \Biggl( \frac{k^{\frac{1}{d-1}}}{4d-4} \Biggr)
		\end{equation}
		by induction on $k$. Note that $\frac{1}{2}\exp \Bigl( \frac{k^{\frac{1}{d-1}}}{4d-4} \Bigr) < 1$ when $k \le (2d-2)^{d-1}$, and so \eqref{eq:Delta'} holds trivially. We assume $k > (2d-2)^{d-1}$ then. For the inductive step, we consider two cases separately: 
		\begin{itemize}
			\item If $m \ge k^{\frac{d-2}{d-1}}$, then from \eqref{eq:stretchbound} and the inductive hypothesis on $d$ we deduce that 
			\[
			\str_P(C) \ge \Delta_{d-1}(m) \ge \frac{1}{2} \exp\Biggl(\frac{m^{\frac{1}{d-2}}}{4d-8}\Biggr) \ge \frac{1}{2} \exp\Biggl(\frac{k^{\frac{1}{d-1}}}{4d-4}\Biggr). 
			\]
			\item If $m \le k^{\frac{d-2}{d-1}}$, then the assumptions $k > (2d-2)^{d-1}$ and $d \ge 3$ imply that $k - m > 2$. Indeed, 
			\[
			k - m \ge k - k^{\frac{d-2}{d-1}} \overset{(*)}{>} (2d-2)^{d-1} - (2d-2)^{d-2} = (2d-2)^{d-2} \cdot (2d-3) \ge 4 \times 3 > 2, 
			\]
			where at the step marked with $(*)$ we used that the function $f_{\alpha}(x) \eqdef x - x^{\alpha}$ is increasing on $(1, +\infty)$ for any $\alpha \in (0, 1)$, which can be seen by taking the derivative of $f_{\alpha}$. Hence, it follows from \eqref{eq:stretchbound} and the inductive hypothesis on $k$ that
			\begin{align*}
				\str_P(C) \ge 2\Delta'_d(k-m) &\ge \exp\Biggl(\frac{(k-m)^{\frac{1}{d-1}}}{4d-4}\Biggr) \ge \exp\left(\frac{k^{\frac{1}{d-1}} \cdot \Bigl(1-k^{-\frac{1}{d-1}}\Bigr)^{\frac{1}{d-1}}}{4d-4}\right) \\
				&\stackrel{(**)}{\ge} \exp\Biggl(\frac{k^{\frac{1}{d-1}} - \frac{2}{d-1}}{4d-4}\Biggr) \ge \frac{1}{2} \exp \Biggl(\frac{k^{\frac{1}{d-1}}}{4d-4}\Biggr), 
			\end{align*}
			where at the step marked with $(**)$ we used the inequality $(1-x)^r \ge 1-2rx$, which is valid for $0\leq r\leq 1$ and $0\leq x\leq \tfrac{1}{2}$, for the pair $(x,r) = \bigl(k^{-\frac{1}{d-1}},\frac{1}{d-1}\bigr)$. 
		\end{itemize}
		By combining the cases above, we conclude \eqref{eq:Delta'}. So, $\Delta'_{d}(k) \ge \frac{1}{2} \exp \Bigl( \frac{k^{\frac{1}{d-1}}}{4d-4} \Bigr)$, and hence 
		\[
		\Delta_d(k) \ge \frac{1}{2} \exp \Biggl( \frac{(k/2)^{\frac{1}{d-1}}}{4d-4} \Biggr) \ge \frac{1}{2} \exp \Biggl( \frac{k^{\frac{1}{d-1}}}{4d} \Biggr)
		\]
		by \Cref{prop:Delta_bound} and the inequality $2^{\frac{1}{d-1}} \le 1+\frac{1}{d-1}$. The inductive proof is complete. 
	\end{proof}
	
	\section{Proof of \texorpdfstring{\Cref{thm:upperbound}}{Theorem 2}} \label{sec:prf_upper}
	The sets that we will construct in the proof of \Cref{thm:upperbound} will be perturbations of lattice points inside suitable regions. Because of that, we start by recalling the basic facts about lattices. An interested reader is referred to \cite[Chapter~12]{pollack} for a more detailed exposition.
	
	Recall that a \emph{lattice} is a discrete subgroup of $\R^d$. Every lattice consists of $\Z$-linear combinations of linearly independent vectors $b_1,\dotsc,b_r$, called a \emph{basis} for the lattice. The \emph{rank} of a lattice is the size of (any one of) its bases. A \emph{fundamental region} of a lattice $\Gamma$ is any set of the form 
	\[
	\{\lambda_1 b_1+\dotsb+\lambda_r b_r : 0\leq \lambda_i\leq 1\text{ for each }i\}, 
	\]
	where $b_1, \dotsc, b_r$ is a basis of~$\Gamma$. The \emph{covolume} of a lattice $\Gamma$ of rank $r$ is the $r$-dimensional Lebesgue measure of (any one of) its fundamental regions. Denote by $\covol(\Gamma)$ the covolume of $\Gamma$. 
	
	\begin{proposition} \label{prop:poly_normvol}
		If $\Gamma$ is a lattice of rank $r$ and $\cP$ is a lattice polytope of $\Gamma$, then $r! \cdot \vol_{\Gamma}(\cP) \in \Z$. Here $\vol_{\Gamma}(P)$ refers to the normalized $r$-dimensional volume of $\cP$ in $\Gamma$, i.e., $\vol_{\Gamma}(\cP) \eqdef \frac{\vol(\cP)}{\covol(\Gamma)}$. 
	\end{proposition}
	
	\begin{proof}
		This is seen by combining \cite[Theorem~3.12]{beck_robins} and \cite[Corollary~3.21]{beck_robins}. 
	\end{proof}
	
	\begin{proposition} \label{prop:covol}
		Suppose $n_1, \dots, n_d, n'$ are integers with $\gcd(n_1, \dots, n_d) = 1$. Define a hyperplane
		\[
		\sigma \eqdef \{ (x_1, \dots, x_d) \in \R^d : n_1x_1 + \dots + n_dx_d = n'\}. 
		\]
		If $\Gamma \eqdef \Z^d \cap \sigma$, then $\Gamma$ is a translate of a lattice whose rank is $d-1$ and covolume is $\sqrt{n_1^2 + \dots + n_d^2}$. 
	\end{proposition}
	
	The proof of \Cref{prop:covol} is technically involved, and we include it in \Cref{append:covol}. 
	
	\smallskip
	Now we are ready to prove \Cref{thm:upperbound}. Let $G \eqdef [n]^d$. The core is to prove the following result: 
	
	\begin{theorem} \label{thm:upperbound_specified}
		For any $d \ge 2$ and any $\veps > 0$, there exists $\widetilde{G} \subset \R^d$ in general position, which is an $\veps$-perturbation of $G$, such that every convex independent subset $\widetilde{C} \subseteq \widetilde{G}$ satisfies $|\widetilde{C}| = O_d\Bigl(n^{\frac{d(d-1)}{d+1}}\Bigr)$. 
	\end{theorem}
	
	\begin{proof}[Proof of \Cref{thm:upperbound} assuming \Cref{thm:upperbound_specified}]
		\Cref{thm:upperbound_specified} directly implies \Cref{thm:upperbound} with a weaker bound of $\sqrt{d}\cdot\sqrt[d]{n}$ in place of~$2\sqrt[d]{n}$. To deduce the stated result, we replace $\Z^d$ by a better lattice.
		
		Consider the densest lattice packing of unit balls in~$\R^d$. Let $\Lambda_d$ be the corresponding lattice, and let $\rho_d$ be its packing density, i.e., $\rho_d \eqdef \lim\limits_{R\to\infty} \frac{\abs{\Lambda_d \cap B_R} \cdot \vol(B_1)}{\vol(B_R)}$. Let $\varphi$ be an invertible affine mapping sending the standard integer lattice $\Z^d$ to~$\Lambda_d$. So, the set $\Lambda_d\cap B_R$ contains $R^d \rho_d \cdot \bigl(1+o(1)\bigr)$ many points, and satisfies $\diam(\Lambda_d\cap B_R)\leq R$. Take $P \eqdef \varphi(\widetilde{G})\cap B_R$. Then \Cref{thm:upperbound} follows from the bound $\rho_d\geq 2^{-d}$ (known as the Minkowski--Hlawka bound, see e.g.~\cite[Theorem 24.1]{gruber_book}). 
	\end{proof}
	
	We remark that no $n$-point set $P$ with $\diam(P) = \alpha \sqrt[d]{n}$ exists when $\alpha < \rho_d^{-\frac{1}{d}}$. This follows easily from the isodiametric inequality (\cite[Theorem 9.2]{gruber_book}) and the definition of~$\rho_d$. Better upper bounds and lower bounds on $\rho_d$ are known, see \cite{venkatesh} and \cite{cohn_zhao} (but they are not enough to improve \Cref{thm:upperbound}). 
	
	\begin{proof}[Proof of \Cref{thm:upperbound_specified}]
		Assume without loss of generality that $0<\veps<\frac{1}{2}$.
		
		The proof consists of three parts: We begin by constructing the perturbation $\widetilde{G}$ of the grid $G$. We then establish a crucial intermediate upper bound concerning the convex independent $\widetilde{C} \subseteq \widetilde{G}$. Finally, we derive the desired upper bound on $|\widetilde{C}|$ with the help of that intermediate bound. 
		
		\vspace{-1em}
		\paragraph{Constructing the perturbation.} By \Cref{thm:oscillator_dup}, there exists a $d$-oscillator $P = \nobreak\{p_1, \dotsc, p_n\}$ with $p_k = (k, p_{k, 2}, \dotsc, p_{k, d})$ such that $P$ is an $\veps$-perturbation of the set $L_0 \eqdef \{(k, 0, \dotsc, 0) : k \in [n]\}$. 
		
		Define, for any $\eta>0$, a linear transformation $T_{\eta}\colon \R^d\to \R^d$ by 
		\[
		T_{\eta}(x_1,x_2,\dotsc,x_d)\eqdef (x_1,\eta x_2,\dotsc,\eta x_d). 
		\]
		Since $T_{\eta}$ is an invertible diagonal linear transformation, $T_{\eta} (P)$ is a $d$-oscillator as well. By \Cref{obs:perturbation}, there exists $\delta > 0$ such that every $\delta$-perturbation of $P$ is a $d$-oscillator as well. Then every $\eta\delta$\nobreakdash-perturbation of $T_{\eta} (P)$ is a $d$-oscillator, as long as $0<\eta< 1$.
		
		For any $x = (x_1, \dotsc, x_d) \in \R^d$, we denote by $\sigma(x) \eqdef (x_d, x_1, \dotsc, x_{d-1})$ the cyclic permutation.
		Choose small constants $\eta_1,\dotsc,\eta_d$ decaying to zero, with $\eta_1 < \frac{\veps}{2}$ and each $\eta_{m+1}$ being much smaller than the preceding~$\eta_m$. 
		The precise conditions will be specified in the proof of \Cref{lem:line_bound}. Set
		\[
		S_i \eqdef \sigma^{i-1}T_{\eta_i}, \qquad G' \eqdef S_1(P) + S_2(P) + \dotsb + S_d(P). 
		\]
		Here by $A+B \eqdef \{a+b : a \in A, \, b \in B\}$ we denote the Minkowski sum of $A, B \subseteq \R^d$. 
		
		View $G'$ as an $\frac{\veps}{2}$-perturbation $f$ of $G$ in the natural way. That is, for each $g = (g_1, \dotsc, g_d) \in G$, 
		\[
		\left(
		\begin{array}{c}
			g_1' \\
			g_2' \\
			g_3' \\
			\vdots \\
			g_d'
		\end{array}
		\right)^{\top}
		= g' \eqdef f(g) = 
		\left(
		\begin{array}{ccccccccc}
			g_1 & + & \eta_2 p_{g_2,d} & + & \eta_3 p_{g_3,d-1} & + & \dotsb & + & \eta_d p_{g_d,2} \\
			\eta_1 p_{g_1,2} & + & g_2 & + & \eta_3 p_{g_3,d} & + & \dotsb & + & \eta_d p_{g_d,3} \\
			\eta_1 p_{g_1,3} & + & \eta_2 p_{g_2,2} & + & g_3 & + & \dotsb & + & \eta_d p_{g_d,4} \\
			&&&&&& \vdots && \\
			\eta_1 p_{g_1,d} & + & \eta_2 p_{g_2,d-1} & + & \eta_3 p_{g_3,d-2} & + & \dotsb & + & g_d
		\end{array}
		\right)^{\top}. 
		\] 
		We may also write this more succinctly as
		\[
		(g_1', g_2', \dotsc, g_d') = S_1(p_{g_1})+ S_2(p_{g_2})+\dotsb+S_d(p_{g_d}).
		\]
		Let $\tau$ be any sufficiently small perturbation of $G'$ such that $\widetilde{G} \eqdef \tau(G')$ is in general position. The precise meaning of `sufficiently small' will be specified later, in the proof of \Cref{lem:line_bound}. For now, we only demand that $\tau$ is an $\frac{\veps}{2}$-perturbation so the composition $\tau\circ f$ is an $\veps$-perturbation.
		
		With $\widetilde{G}$ being defined, our goal is to establish the upper bound on $|\widetilde{C}|$. For brevity, we write
		\[
		\bullet' = f(\bullet), \ \ \bullet = f^{-1}(\bullet'); \qquad \widetilde{\bullet} = \tau(\bullet'), \ \ \bullet' = \tau^{-1}(\widetilde{\bullet}). 
		\]
		Here $\bullet$, $\bullet'$, and $\widetilde{\bullet}$ are subsets of $G$, $G'$, and $\widetilde{G}$, respectively. 
		
		Consider the preimage $C \subseteq G$ of $\widetilde{C}$ under~$\tau\circ f$. Notice that $C = f^{-1} \bigl( \tau^{-1}(\widetilde{C}) \bigr)$. 
		
		Assume without loss of generality that $\veps$ is small enough in terms of $n$ such that the interior of $\conv(C)$ contains no point of $C$, i.e., $\conv_0(C) \cap C = \varnothing$. 
		
		Set $\bC \eqdef \conv(C)$. Then $\bC$ is a lattice polytope (possibly degenerate, i.e., of a lower dimension) with vertices in $G$. If $\bC$ is degenerate, then we take the intersection of an arbitrary $(d-1)$-hyperplane containing $\bC$ and the cube $[1, n]^d$ as its only facet. For each facet $\bF$ of $\bC$, write $F \eqdef \bF \cap \Z^d$ and $\Gamma_F \eqdef \spa(F) \cap \Z^d$. Recall that $\spa(F)$ is the affine span of $F$, which is the unique hyperplane containing $F$. By \Cref{prop:covol}, $\Gamma_F$ is a translate of a lattice of rank $d - 1$. Let $n_F \eqdef \covol(\Gamma_F)$. 
		
		Denote by $\area(\bullet)$ the $(d-1)$-dimensional Lebesgue measure. For any finite subset $S \subset \Gamma_F$, we define the normalized $(d-1)$-dimensional volume of $\conv(S)$ as $A_S \eqdef \frac{\area(\conv(S))}{n_F}$. Let $\cF$ be the set of all $F=\bF \cap \Z^d$, as $\bF$ ranges over all facets of~$\bC$. 
		
		\vspace{-1em}
		\paragraph{Establishing the intermediate bound.} We prove for all $F \in \cF$ the upper bound
		\begin{equation} \label{eq:face_bound}
			|C \cap F| = O_d \Bigl( A_{F}^{\frac{d-2}{d-1}} \log A_F \Bigr). 
		\end{equation}
		Call the set $L = \{a+kv : k \in [\ell]\}$ a \emph{lattice line segment} if $a, v \in \Z^d$ and $\ell \ge 2$. To prove \eqref{eq:face_bound}, we first analyze the behavior of $|C \cap L|$, and then pass to $|C \cap F|$ by a discrete John-type result. Such results can be found in \cite{tao_vu,berg_henk,hintum_keevash}. For our purpose, a weak version in \cite{barany_vershik} is enough.
		
		\begin{lemma} \label{lem:line_bound}
			Suppose $L \subset G$ is a lattice line segment. Then $|C \cap L| = O_d(\log|L|)$. 
		\end{lemma}
		
		\begin{proof}
			Suppose $|L| = \ell$ and write $L = \{a+kv : k \in [\ell]\}$, where $a = (a_1, \dotsc, a_d)$ and $v = (v_1, \dotsc, v_d)$. Let $m \in [d]$ be the smallest index such that $v_m \neq 0$ (so $v_1 = \dotsb = v_{m-1} = 0$). 
			
			The proof uses two affine transformations $\psi$ and $\psi'$ that are very close to each other. In fact, $\psi$ takes $L$ into a line segment that is parallel to the first coordinate axis; $\psi'$ maps $L'=f(L)$ to a small perturbation of $\psi(L)$, which itself is a small perturbation of $T_{\eta_m}(P)$. These properties imply that $\psi'(L')$ is a $d$-oscillator. An application of \Cref{thm:expstretch} then completes the proof.
			
			Recall that $p_k = (k, p_{k,2}, \dotsc, p_{k,d})$. We begin by defining $\psi$ and $\psi'$ in three steps.
			\vspace{-0.5em}
			\begin{itemize}
				\item Translate $L$ and $L'$ via 
				\begin{align*}
					\lambda(x)&\eqdef x-(a_1,\dotsc,a_{m-1},0,a_{m+1},\dotsc,a_d),\\
					\lambda'(x)&\eqdef x-\bigl(S_1(p_{a_1}) + \dotsb +  S_{m-1}(p_{a_{m-1}}) + S_{m+1}(p_{a_{m+1}}) + \dotsb + S_d(p_{a_d})\bigr),
				\end{align*}
				respectively. Note that the term $S_m(p_{a_m})$ is absent in the definition of $\lambda'(x)$. 
				\vspace{-0.5em}
				\item Align the line segment $\lambda(L)$ with the first coordinate axis via
				\[
				\varphi(x_1,\dotsc,x_m,x_{m+1},\dotsc,x_d)\eqdef (x_1,\dotsc,x_m,x_{m+1}-\alpha_{m+1}x_m,\dotsc,x_d-\alpha_dx_m), 
				\]
				where $\alpha_t\eqdef\frac{v_t}{v_m} \, (t = m+1, \dots, d)$. We use the same $\varphi$ to align $\lambda'(L')$. 
				\vspace{-0.5em}
				\item Construct transformations $\psi$ and $\psi'$ of $L$ and $L'$, respectively, by
				\begin{align*}
					\psi &\eqdef \sigma^{1-m} \circ \varphi \circ \lambda, \\
					\psi' &\eqdef \sigma^{1-m} \circ \varphi \circ \lambda'. 
				\end{align*}
			\end{itemize}
			\vspace{-0.5em}
			Observe that $\psi(L)$ is a line segment on the first coordinate axis, whereas $\psi'(L')$ consists of points
			\[
			T_{\eta_m}(p_i) + V_{n, d}(\eta_{m+1})
			\]
			as integer $i$ runs over the arithmetic progression $\{a_m + kv_m : k \in [\ell]\}$. Here $V_{n, d}(\eta_{m+1})$ is a vector of norm $O_{n, d}(\eta_{m+1})$. Let $\hat{L}$ consist of points of the form $T_{\eta_m}(p_i)$ as $i$ ranges over the same progression. From \Cref{prop:oscillator} applied to the $d$-oscillator $T_{\eta_m} (P)$, we see that $\hat{L} \subset T_{\eta_m}(P)$ is also a $d$-oscillator. 
			
			Recall that every $\delta$-perturbation of the $d$-oscillator $P$ is a $d$-oscillator, and every $\eta\delta$-perturbation of $T_{\eta}(P)$ is a $d$-oscillator. It is easily seen that every $\eta_m\delta$-perturbation of $\hat{L}$ can be extended to an $\eta_m\delta$-perturbation of $T_{\eta_m}(P)$ (e.g., perturb without moving any point from $T_{\eta_m} \setminus \hat{L}$). By applying \Cref{prop:oscillator} to the perturbed $T_{\eta_m}(P)$, we see that every $\eta_m\delta$-perturbation of $\hat{L}$ is a $d$-oscillator. 
			
			Choose $\eta_1 < \frac{\veps}{2}$ and $\eta_{i+1}$ in terms of $\eta_i$ for $i = 1, \dots, d-1$ so that the norm of the vector $V_{n, d}(\eta_{i+1})$ satisfies $|V_{n, d}(\eta_{i+1})| < \eta_i \delta$. Then $\psi'(L')$, being an $\eta_m \delta$-perturbation of $\hat{L}$, is a $d$-oscillator. Thus, we can choose the perturbation $\tau$ to be small enough so that $\psi'(\widetilde{L})$ is a sufficiently small perturbation of $\psi'(L')$ and hence again a $d$-oscillator, as per \Cref{obs:perturbation}. 
			
			Finally, since $\psi'(\widetilde{C}\cap \widetilde{L})$ is a convex subset of the $d$-oscillator $\psi'(\widetilde{L})$, from \Cref{thm:expstretch} we deduce that $\abs{\psi'(\widetilde{C}\cap \widetilde{L})}=O_d\bigl(\log \abs{\psi'(\widetilde{L})}\bigr)=O_d(\log |L|)$, and hence 
			$\abs{C\cap L}=\abs{\psi'(\widetilde{C}\cap \widetilde{L})}=O_d(\log |L|)$. 
		\end{proof}
		
		\begin{proof}[Proof of estimate \eqref{eq:face_bound}]
			Suppose $B$ is a basis for a lattice~$\Gamma$. Consider $\Z^B$, i.e., integer vectors indexed by~$B$. For any $\alpha, \beta \in \Z^B$, we define 
			\[
			P(B; \alpha, \beta) \eqdef \Biggl\{ \sum_{b\in B} \lambda_b b : \lambda_b \in [\alpha_b, \beta_b] \cap \Z \text{ for each }b\in B \Biggr\}. 
			\]
			For example, if $\Gamma = \Z^d$ and $B = \{e_1, \dots, e_d\}$ is the standard basis, then $\conv\bigl(P(B; \alpha, \beta)\bigr)$ is an axis-parallel box in $\R^d$ with integer vertices.
			We thus call $P(B; \alpha, \beta)$ a \emph{lattice parallelotope} in~$\Gamma$. 
			
			Fix an arbitrary basis $B = \{b_1, \dotsc, b_{d-1}\}$ of $\Gamma_F$. There exists a lattice parallelotope $P = P_F$ such that $F \subseteq P \subset \Gamma_F$ and $A_P = O_d(A_F)$, thanks to \cite[Theorem 3]{barany_vershik}. 
			
			Suppose $\alpha, \beta \in \Z^{d-1}$ are such that
			\[
			P = \Biggl\{ \sum_{i=1}^{d-1} \lambda_i b_i : \lambda_i \in [\alpha_i, \beta_i] \cap \Z \Biggr\}. 
			\]
			Then $A_P = \prod_{i=1}^{d-1} (\beta_i - \alpha_i)$. Since $A_F > 0$, and so $A_P > 0$, we have $\beta_i > \alpha_i$ for all $i$. Let $\ell_i \eqdef \beta_i - \alpha_i$; note that $\ell_i\geq 1$ because $\ell_i\in \Z$. Then $A_P = \prod_{i=1}^{d-1} \ell_i$. Without loss of generality, we may assume that $\ell_1 = \max\{\ell_1, \dotsc, \ell_{d-1}\}$. This implies
			\begin{equation} \label{eq:para_vol_dup}
				\prod_{i=2}^{d-1} (\ell_i+1) \le \prod_{i=2}^{d-1} (2\ell_i) \le 2^{d-1} \cdot \prod_{i=2}^{d-1} \Bigl( \ell_1^{\frac{1}{d-1}} \ell_i^{\frac{d-2}{d-1}} \Bigr) = 2^{d-1} \cdot A_P^{\frac{d-2}{d-1}}. 
			\end{equation}

			Partition $P$ into $\prod_{i=2}^{d-1} (\ell_i+1)$ lattice line segments that are parallel to the basis vector $b_1$. Then each of these line segments contains exactly $\ell_1+1$ points of~$\Z^d$.
			\Cref{lem:line_bound} implies that each such lattice segment $L$ satisfies $|C \cap L| = O_d(\log |L|) = O_d(\log \ell_1)$, and so 
			\[
			|C \cap F| \le |C \cap P| = \prod_{i=2}^{d-1} (\ell_i+1) \cdot O_d(\log \ell_1) \overset{(*)}{=} O_d \Bigl( A_P^{\frac{d-2}{d-1}} \Bigr) \cdot O_d(\log A_P) \overset{(**)}{=} O_d\Bigl(A_F^{\frac{d-2}{d-1}} \log A_F\Bigr), 
			\]
			where $(*)$ follows from \eqref{eq:para_vol_dup} and $\ell_1 \le \prod_{i=1}^{d-1} \ell_i = A_P$, whereas $(**)$ is implied by $A_P = O_d(A_F)$. 
		\end{proof}
		
		\vspace{-1em}
		\paragraph{Finishing the proof.} We need another technical lemma on $n_F$ to establish \Cref{thm:upperbound_specified}. 
		
		\begin{proposition} \label{prop:normvol_est}
			Let $\nu_k$ be the $k$-th smallest element of the multiset $\cN \eqdef \{n_F : F \in \cF\}$. Then we have $\nu_k = \Omega\bigl( k^{\frac{1}{d}} \bigr)$ for $k = 1, \dots, |\cF|$. 
		\end{proposition}
		
		\begin{proof}
			Define an auxiliary multiset
			\[
			\mathcal{U} \eqdef \biggl\{\sqrt{n_1^2+\dotsb+n_d^2} : (n_1, \dotsc, n_d) \in \Z^d, \, \gcd(n_1, \dotsc, n_d) = 1 \biggr\}. 
			\]
			Let $u_k$ be the $k$-th smallest element of $\mathcal{U}$. Then $u_1 = \dotsb = u_{2d} = 1, \, u_{2d+1} = \dotsb = u_{2d^2} = \sqrt{2}$. The reduced outer normal vectors (relative to $\bC$) of $\spa(F)$ are distinct for $F \in \cF$. \Cref{prop:covol} then implies that $\nu_k \ge u_k$ for $k = 1, \dots, |\cF|$. Hence, it suffices to establish $u_k = \Omega\bigl( k^{\frac{1}{d}} \bigr)$ for all $k \ge 1$. 
			
			To that end, we note that $r = \frac{1}{3}k^{\frac{1}{d}}$ and $k$ being sufficiently large together imply that
			\[
			|B_r \cap \Z^d| \le |(-r, r)^d \cap \Z^d| < (2r+1)^d < k.
			\]
			From this it follows that $u_k \ge r = \frac{1}{3} k^{\frac{1}{d}}$.
		\end{proof}
		
		\smallskip
		
		Finally, we are ready to show that $|\widetilde{C}| = O_d\Bigl( n^{\frac{d(d-1)}{d+1}} \Bigr)$. 
		
		\smallskip
		
		Write $A \lesssim B$ if $A = O_{d}(B)$. We claim that 
		\begin{equation} \label{eq:surfacearea_dup}
			\sum_{F \in \cF} A_F n_F \lesssim n^{d-1}. 
		\end{equation}
		If $\bC$ is degenerate, then $\bC$ has a unique facet $F_0$, which is the intersection of the cube $[1, n]^d$ and some $(d-1)$-dimensional hyperplane. The definitions of $A_{F_0}$ and $n_{F_0}$ imply that $A_{F_0} n_{F_0}$ is the area of $F_0$, which is obviously upper bounded by the surface area of $[1, n]^d$. It follows that
		\[
		\sum_{F \in \cF} A_F n_F = A_{F_0} n_{F_0} \le \area\bigl(\partial [1, n]^d\bigr) = 2d(n-1)^{d-1} \lesssim n^{d-1}. 
		\]
		Suppose $\bC$ is non-degenerate then. Since both $\bC$ and $[1, n]^d$ are convex and $\bC \subseteq [1, n]^d$, it follows from the monotonicity of intrinsic volumes (\cite[Theorem~6.13(iv)]{gruber_book}) that
		$\area\bigl(\partial \bC\bigr) \le \area\bigl(\partial [1, n]^d\bigr)$. Since $A_Fn_F$ is the area of $\bF=\conv(F)$ and $2d(n-1)^{d-1}$ is the area of $\partial [1, n]^d$, we have
		\[
		\sum_{F \in \cF} A_F n_F = \area\bigl(\partial \bC\bigr) \le \area\bigl(\partial [1, n]^d\bigr) = 2d(n-1)^{d-1} \lesssim n^{d-1}. 
		\]
		
		From \Cref{prop:normvol_est} combined with \eqref{eq:surfacearea_dup} we deduce that 
		\[
		|\cF|^{\frac{d+1}{d}} \lesssim \int_{0}^{|\cF|}x^{\frac{1}{d}}\,dx = \sum_{k=1}^{|\cF|} \int_{k-1}^k x^{\frac{1}{d}} \, dx \le \sum_{k=1}^{|\cF|} k^{\frac{1}{d}} \lesssim \sum_{F \in \cF} n_F \overset{(*)}{\le} d!\sum_{F \in \cF} A_F n_F \lesssim n^{d-1},
		\]
		where at $(*)$ we used the fact $d! A_F \ge (d-1)! A_F \ge 1$, a corollary of \Cref{prop:poly_normvol}. (Notice that $\bF$ is a $(d-1)$-dimensional polytope, no matter $\bC$ is degenerate or non-degenerate.) It follows that
		\begin{equation} \label{eq:f_bound_dup}
			|\cF| \lesssim n^{\frac{d(d-1)}{d+1}}. 
		\end{equation}
		
		Fix any real number $\lambda \in \bigl(\frac{d-2}{d-1}, \frac{d}{d+1}\bigr)$. Set $p \eqdef \lambda^{-1}$ and $q \eqdef (1-\lambda)^{-1}$. It follows from $0<\lambda<\frac{d}{d+1}$ that $0 < \frac{q\lambda}{d} < 1$, and hence 
		\begin{equation} \label{eq:sum_bound}
			\Biggl(\sum_{k=1}^{|\cF|} k^{-\frac{q\lambda}{d}}\Biggr)^{\frac{1}{q}} \le \Biggl( \sum_{k=1}^{|\cF|} \int_{k-1}^k x^{-\frac{q\lambda}{d}} \, dx \Biggr)^{\frac{1}{q}} = \biggl( \int_0^{|\cF|} x^{-\frac{q\lambda}{d}} \, dx \biggr)^{\frac{1}{q}} \lesssim \Bigl(|\cF|^{1 - \frac{q\lambda}{d}}\Bigr)^{\frac{1}{q}} = |\cF|^{1-\lambda\frac{d+1}{d}}. 
		\end{equation}
		Thus, we obtain
		\begin{align*}
			|C| &\le \sum_{F \in \cF} |C \cap F| && \text{because of $\conv_0(C) \cap C = \varnothing$} \\
			&\lesssim \sum_{F \in \cF} A_F^{\lambda} = \sum_{F \in \cF} (A_Fn_F)^{\lambda} \cdot n_F^{-\lambda} && \text{using \eqref{eq:face_bound} and $\lambda > \tfrac{d-2}{d-1}$} \\
			&\le \Biggl( \sum_{F \in \cF} (A_Fn_F)^{p\lambda} \Biggr)^{\frac{1}{p}} \cdot \Biggl( \sum_{F \in \cF} n_F^{-q\lambda} \Biggr)^{\frac{1}{q}} && \text{by H\"{o}lder's inequality} \\
			&\lesssim \Biggl( \sum_{F \in \cF} A_Fn_F \Biggr)^\frac{1}{p} \cdot \Biggl( \sum_{k=1}^{|\cF|} k^{-\frac{q\lambda}{d}} \Biggr)^{\frac{1}{q}} && \text{by \Cref{prop:normvol_est} and $-q\lambda < 0$} \\
			&\lesssim \bigl( n^{d-1} \bigr)^{\lambda} \cdot \Bigl(|\cF|^{1-\lambda\frac{d+1}{d}}\Bigr) && \text{using \eqref{eq:surfacearea_dup} and \eqref{eq:sum_bound}} \\
			&\lesssim \Bigl( n^{\frac{d(d-1)}{d+1}} \Bigr)^{\lambda\frac{d+1}{d}} \cdot \Bigl( n^{\frac{d(d-1)}{d+1}} \Bigr)^{1-\lambda\frac{d+1}{d}} &&\text{using \eqref{eq:f_bound_dup} and $\lambda < \tfrac{d}{d+1}$} \\
			&= n^{\frac{d(d-1)}{d+1}}. 
		\end{align*}
		Since $|\widetilde{C}| = |C|$, the proof of \Cref{thm:upperbound_specified} is complete. 
	\end{proof}

\appendix 
\section*{Appendix}
\section{The number of bad points is finite} \label{append:bad}
	
	Observe that $\pi(P_{i+1}')$ is automatically regular since $P^*$ is regular. From \ref{regular:axis_inj_dup} and \ref{regular:P_gen_dup} we deduce that there are three kinds of possible reasons behind a point $p' \in \ipi(p_{i+1}^*)$ being bad: 
	\vspace{-0.5em}
	\begin{enumerate}[label=(B\arabic*), ref=(B\arabic*)]
		\item \label{bad:height} The height of $p'$ happens to be the same as some $p_j'$ for $j = 1, \dots, i$. 
		\vspace{-0.5em}
		\item \label{bad:general} The point $p'$ lies on the $(d-1)$-dimensional hyperplane spanned by points $p_{j_1}', \dots, p_{j_d}' \in P_i'$. 
		\vspace{-0.5em}
		\item \label{bad:generic} There exist disjoint subsets $S, T$ of $P_i'$ with $|S| \ge 1, \, |T| \ge 1$ and $|S| + |T| = d+1$ such that 
		\[
		\lvert \spa(S) \cap \spa(T \cup \{p'\}) \rvert \neq 1. 
		\]
		We do not include $|T| = 0$ because in that case $\spa(S) = \R^d$ and $\lvert \spa(S) \cap \spa(\{p'\}) \rvert = 1$. 
	\end{enumerate}
	\vspace{-0.5em}
	
	Firstly, any bad point from \ref{bad:height} is (for some $j$) the intersection of the horizontal hyperplane 
	\[
	\{x \in \R^d : h(x) = h(p_j')\}
	\]
	and the vertical line $\ipi(p_{i+1}^*)$. Obviously, the number of such bad points is finite. 
	
	Secondly, we claim that the hyperplane $\spa(\{p_{j_1}', \dots, p_{j_d}'\})$ is not parallel to $e_d = (0, \dots, 0, 1)$. This is seen by the fact that $\pi(p_{j_1}'), \dots, \pi(p_{j_d}')$ span $\R^{d-1}$, since $\pi(P_i) \subset \R^{d-1}$ is regular hence in general position. The claim then implies that the number of bad points from \ref{bad:general} is finite. 
	
	Thirdly, suppose $|S| \eqdef s, \, |T| \eqdef t$, and $S = \{u_1, \dots, u_s\}, \, T \cup \{p'\} = \{v_1, \dots, v_t, v_{t+1} = p'\}$. Then $\spa(S) = u_1 + U, \, \spa(T) = v_1 + V$, where $U$ and $V$ denote the linear subspaces of $\R^d$ generated by $\wu_2 \eqdef u_2-u_1, \dots, \wu_s \eqdef u_s-u_1$ and $\wv_2 \eqdef v_2-v_1, \dots, \wv_t \eqdef v_t-v_1$, respectively. Also, we denote $\wv_{t+1} \eqdef v_{t+1} - v_1$ and $V' \eqdef \wv_{t+1} + V$. It follows from $P_i'$ is regular (hence in general position) that $\dim(U) = s-1$ and $\dim(V) = t-1$. Moreover, since $P^*$ is regular, and so in general position, we deduce that $\wv_{t+1} \notin \spa(\wv_1, \dots, \wv_t)$. So, by noticing $\dim(U) + \dim(V') = (s-1) + t = d$ we obtain
	\begin{equation} \label{eq:generic}
		\bigl\lvert \spa(S) \cap \spa(T \cup \{p'\}) \bigr\rvert = 1 \iff U \cap V' = \{0\} \iff \R^d = U \oplus V'. 
	\end{equation}
	Similarly, from $P^*$ is regular we deduce that $\dim\bigl(\pi(U)\bigr) = s-1, \, \dim\bigl(\pi(V)\bigr) = t-1$, and hence
	\begin{equation} \label{eq:pi_generic}
		\bigl\lvert \spa\bigl(\pi(S)\bigr) \cap \spa\bigl(\pi(T)\bigr) \bigr\rvert = 1 \iff \pi(U) \cap \pi(V) = \{0\} \iff \R^{d-1} = \pi(U) \oplus \pi(V). 
	\end{equation}
	Let $W$ be the linear subspace $\spa(\wu_2, \dots, \wu_s, \wv_2, \dots, \wv_t)$. Since $P_i'$ is regular, from \eqref{eq:generic} we see that
	\[
	\text{$p'$ is bad because of \ref{bad:generic}} \iff \R^d \neq U \oplus V' \iff \wv_{t+1} \in W \iff p' \in v_1 + W. 
	\]
	It then suffices to prove that $e_d \notin W$, as this implies that the number of bad points from \ref{bad:generic} is finite, which completes the proof. Suppose to the contrary that $e_d \in W$. Upon taking the projection $\pi$, the $d-1$ vectors $\pi(\wu_2), \dots, \pi(\wu_s), \pi(\wv_2), \dots, \pi(\wv_t) \in \R^{d-1}$ are linearly dependent. This implies that $\pi(U) \oplus \pi(V) \neq \R^{d-1}$, a contradiction to \eqref{eq:pi_generic}. 
	
	\section{Proof of \texorpdfstring{\Cref{prop:covol}}{Proposition~19}} \label{append:covol}
	\paragraph{Proof that \texorpdfstring{$\Gamma$}{Gamma} is a translate of a lattice of rank \texorpdfstring{$d-1$}{d-1}.}
	The vector $\hat{n} \eqdef (n_1, \dots, n_d)$ is a normal vector to the hyperplane $\sigma$. Note that $|\hat{n}| = \sqrt{n_1^2 + \dots + n_d^2}$. We recall the definition of $\Gamma$ and define an auxiliary lattice $\Gamma_0$ of rank $d-1$ as follows: 
	\begin{align*}
		\Gamma &\eqdef \{ x \in \Z^d : \langle \hat{n}, x \rangle = n'\}, \\
		\Gamma_0 &\eqdef \{ x \in \Z^d : \langle \hat{n}, x \rangle = 0\}. 
	\end{align*}
	Since $\gcd(n_1, \dots, n_d) = 1$, there exists $y = (y_1, \dots, y_d) \in \Z^d$ such that $\langle \hat{n}, y \rangle = n_1y_1 + \dots + n_dy_d = 1$. This implies that $z \eqdef n'y \in \Gamma$. We claim that $\Gamma = \Gamma_0 + z$. To see this, we observe that
	\begin{itemize}
		\item $\langle \hat{n}, u-z \rangle = 0$ holds for any $u \in \Gamma$, which implies $u-z \in \Gamma_0$; and
		\item $\langle \hat{n}, v+z \rangle = n'$ holds for any $v \in \Gamma_0$, which implies $v+z \in \Gamma$.
	\end{itemize}
	Because $\Gamma = \Gamma_0 + z$, it follows that $\Gamma$ is a translate of a lattice of rank $d-1$.
	
	\vspace{-1em}
	\paragraph{Proof that \texorpdfstring{$\covol(\Gamma_0)=|\hat{n}|.$}{covol(Gamma\_0)=|n-hat|}}
	We may assume without loss of generality that $n' = 0$, and so $\Gamma = \Gamma_0$.
	
	We first establish that $\covol(\Gamma) \ge |\hat{n}|$. As a lattice of rank $d-1$, suppose that $\Gamma$ is generated by basis vectors $b_1, \dots, b_{d-1}$. The strategy is to define an integer vector $\hat{b} \in \Z^d$ which is orthogonal to $b_1, \dots, b_{d-1}$, hence parallel to $\hat{n}$. Then the condition $\gcd(n_1, \dots, n_d) = 1$ implies that $|\hat{n}|$ divides $|\hat{b}|$, and hence $|\hat{b}| \ge |\hat{n}|$. Finally, we conclude by establishing $\covol(\Gamma) = |\hat{b}|$.
	
	Think about $b_1, \dots, b_{d-1}$ as row vectors and write $b_i \eqdef (b_{i, 1}, \dots, b_{i, d-1})$ for $i = 1, \dots, d-1$. Set
	\[
	B \eqdef 
	\begin{pmatrix}
		b_1 \\
		\vdots \\
		b_{d-1}
	\end{pmatrix}
	= 
	\begin{pmatrix}
		b_{1,1} & \cdots & b_{1,d} \\
		\vdots & \ddots & \vdots \\
		b_{d-1,1} & \cdots & b_{d-1,d}
	\end{pmatrix}. 
	\]
	For $i = 1, \dots, d$, let $B_i$ be the $(d-1) \times (d-1)$ matrix obtained by deleting the $i$-th column of $B$. Set $\lambda_i \eqdef (-1)^{i+1} \det(B_i)$ and $\hat{b} \eqdef \lambda_1 e_1 + \dots + \lambda_d e_d$, where $e_1, \dots, e_d$ are the standard basis vectors. We remark that the vector $\hat{b}$ constructed is the \emph{exterior product} of $b_1, \dots, b_{d-1}$.
	
	We claim that $\hat{b}$ is orthogonal to $b_i$ for each $i = 1, \dots, d-1$. Indeed, we have
	\[
	\langle \hat{b}, b_i \rangle = \lambda_1 b_{i, 1} + \dots + \lambda_d b_{i, d} = \det
	\begin{pmatrix}
		b_{i, 1} & \cdots & b_{i, d} \\
		b_{1, 1} & \cdots & b_{1, d} \\
		\vdots & \ddots & \vdots \\
		b_{d-1, 1} & \cdots & b_{d-1, d}
	\end{pmatrix}
	= 0. 
	\]
	It then suffices to show that $\covol(\Gamma) = |\hat{b}|$. To see this, we compute
	\[
	\covol(\Gamma) \overset{(*)}= \sqrt{\det(BB^{\top})} \overset{(**)}{=} \sqrt{\sum_{i=1}^d \det(B_iB_i^{\top})} = \sqrt{\sum_{i=1}^d \det(B_i)^2} = \sqrt{\sum_{i=1}^d \lambda_i^2} = |\hat{b}|, 
	\]
	where at the step $(**)$ we used the Cauchy--Binet formula. As for the step $(*)$, it can be deduced from basic Riemannian Geometry (see \cite[Proposition~15.31]{lee}, for instance). Here we also include its elementary proof: We begin by proving $(*)$ in the special case when $b_1, \dots, b_{d-1} \in \spa(e_1, \dots, e_{d-1})$ and hence $b_{1, d} = \dots = b_{d-1, d} = 0$. Recall that $B_d$ is the $(d-1) \times (d-1)$ matrix
	\[
	B_d =
	\begin{pmatrix}
		b_{1,1} & \cdots & b_{1,d-1} \\
		\vdots & \ddots & \vdots \\
		b_{d-1,1} & \cdots & b_{d-1,d-1}
	\end{pmatrix}. 
	\]
	Write $\bm{0} \eqdef (0, \dots, 0) \in \Z^d$. Then the $(d-1) \times d$ matrix $B$ is equal to $(B_d \ \bm{0}^{\top})$, and hence we obtain
	\[
	\covol(\Gamma) = \abs{\det(B_d)} = \sqrt{\det(B_d^{} B_d^{\top})} = \sqrt{\det\bigl( (B_d \ \bm{0}^{\top}) (B_d \ \bm{0}^{\top})^{\top} \bigr)} = \sqrt{\det(BB^{\top})}.
	\]
	For general vectors $b_1, \dotsc, b_{d-1} \in \Z^d$, pick any orthogonal transformation $M$ sending $b_1, \dots, b_{d-1}$ to the subspace $\spa(e_1, \dots, e_{d-1})$. Then
	\begin{align*}
		\covol(\Gamma) = \covol(\Gamma M) &= \sqrt{\det \bigl( (BM) (BM)^{\top} \bigr)} = \sqrt{\det(BB^{\top})}. 
	\end{align*}
	
	\smallskip
	
	We next show that $\covol(\Gamma) \le |\hat{n}|$. The strategy is to establish that $\covol(\Gamma)$ divides $n_i^{d-2} |\hat{n}|$ for every $i = 1, \dots, d$. It follows from $\gcd(n_1, \dots, n_d) = 1$ that $\gcd\bigl(n_1^{d-2}, \dots, n_d^{d-2}\bigr) = 1$, and so there exist $m_1, \dots, m_d \in \Z$ such that $m_1n_1^{d-2} + \dots + m_dn_d^{d-2} = 1$. This then implies that $\covol(\Gamma)$ divides
	\[
	m_1n_1^{d-2}|\hat{n}| + \dots + m_dn_d^{d-2}|\hat{n}| = \bigl(m_1n_1^{d-2} + \dots + m_dn_d^{d-2}\bigr) \cdot |\hat{n}| = |\hat{n}|. 
	\]
	We thus obtain $\covol(\Gamma) \le |\hat{n}|$. 
	
	It then suffices to verify that $\covol(\Gamma)$ divides $n_1^{d-2}|\hat{n}|$, since by symmetry this will imply that $\covol(\Gamma)$ divides each $n_i^{d-2}|\hat{n}|$. For $j = 2, \dots, d$, consider the vector $v_j \eqdef n_je_1 - n_1e_j$. It is easily seen that $v_j$ is orthogonal to $\hat{n}$. Notice that $v_j\in \Gamma$ (since we assumed $n' = 0$). Let $\cP$ be the parallelotope generated by $v_2, \dots, v_d$. Since the vectors generating $\cP$ are in $\Gamma$, we obtain $\covol(\Gamma)$ divides $\vol(\cP)$, the $(d-1)$-dimensional volume of $\cP$. Let $\cP'$ be the $d$-dimensional parallelotope obtained by extruding $\cP$ by the vector $\hat{n}$, and denote by $\vol(\cP')$ its $d$-dimensional volume. Then
	\[
	|\hat{n}| \cdot \vol(\cP) = \vol(\cP') = \left\lvert \det 
	\begin{pmatrix}
		n_1 & n_2 & n_3 & \cdots & n_d \\
		n_2 & -n_1 & 0 & \cdots & 0 \\
		n_3 & 0 & -n_1 & \cdots & 0 \\
		\vdots & \vdots & \vdots & \ddots & \vdots \\
		n_1 & 0 & 0 & \cdots & -n_1
	\end{pmatrix}
	\right\rvert = |n_1|^{d-2} (n_1^2 + \dots + n_d^2), 
	\]
	where the determinant is computed inductively by expanding along the last column. From the fact $|\hat{n}|^2 = n_1^2 + \dots + n_d^2$ we deduce that $\vol(\cP) = |n_1|^{d-2}|\hat{n}|$. Thus, $\covol(\Gamma)$ divides $|n_1|^{d-2}|\hat{n}|$.
	
	\smallskip
	
	The proof of \Cref{prop:covol} is complete. 

\section*{Acknowledgments} 
We thank Adrian Dumitrescu and Csaba D.\ T\'oth for useful comments about the previous version of the paper, and in particular for alerting us of an error therein. The parameter $\lambda$ in the proof of \Cref{thm:upperbound_specified} was not in the interval $\bigl(\frac{d-2}{d-1}, \frac{d}{d+1}\bigr)$. The second author thanks Yisai Xue for assistance during the revision of this paper, and Minghui Ouyang for helpful discussions on \Cref{append:covol}. 

\bibliographystyle{plain}
\bibliography{res_es}

\begin{aicauthors}
\begin{authorinfo}[bukh]
  Boris Bukh\\
  Department of Mathematical Sciences, Carnegie Mellon University\\
  Pittsburgh, USA\\
  bbukh\imageat{}math\imagedot{}cmu\imagedot{}edu \\
  \url{http://www.borisbukh.org}
\end{authorinfo}
\begin{authorinfo}[dong]
  Zichao Dong\\
  Extremal Combinatorics and Probability Group (ECOPRO), Institute for Basic Science (IBS)\\
  Daejeon, South Korea\\
  zichao\imageat{}ibs\imagedot{}re\imagedot{}kr \\
  \url{https://dzch0310.github.io/dzch0310/}
\end{authorinfo}
\end{aicauthors}

\end{document}